\documentclass[oneside,reqno]{amsart}
\usepackage[margin=3cm]{geometry}

\usepackage{graphicx}
\usepackage{geometry}
\usepackage{graphicx}
\usepackage{amsmath,amssymb,amsthm}
\usepackage{todonotes}
\usepackage{hyperref}
\usepackage{enumitem}
\usepackage{mathrsfs}
\usepackage{bm}

\newtheorem{theorem}{Theorem}
\newtheorem{lemma}[theorem]{Lemma}
\newtheorem{corollary}[theorem]{Corollary}

\newtheorem{proposition}[theorem]{Proposition}
\newtheorem{remark}[theorem]{Remark}

\theoremstyle{definition}

\theoremstyle{plain}

\title{Point Processes and Gap Statistics for Translation Surfaces}

\author{Kai Fu}
\address[Kai Fu]{Max Planck Institute for Mathematics in the Sciences\\
Inselstrasse 22, A3 11, Leipzig 04103, Germany}
\email{kai.fu@mis.mpg.de}

\author{Anthony Sanchez}
\address[Anthony Sanchez]{Department of Mathematics, University of Houston\\ Philip Guthrie Hoffman Hall
3551 Cullen Blvd, Room 641
Houston, Texas 77204, USA}
\email{asanchez109@uh.edu}

\date{October 2026}

\begin{document}

\begin{abstract}
We study the behavior of saddle connection directions on translation surfaces from the point process viewpoint. We define a direction point process for saddle connections of length at most $R>0$ on every translation surface. Given a general ergodic $\mathrm{SL}(2,\mathbb R)$-invariant measure $\mu$ on a stratum of translation surfaces, we also define a slope point process associated with $\mu$. By utilizing the dynamics and geometry of strata, we show that for almost every translation surface, the direction point processes, rescaled by $R^2$, converge in distribution to the slope point process as $R\to\infty$. Using Palm distribution theory, we prove the existence of limiting joint gap distributions and obtain new regularity results for gap distributions associated with the general $\mu$. Finally, we study how these point processes and joint gap distributions behave as the genus $g$ tends to infinity. In the principal stratum, we show that the slope point processes, rescaled by $g^2$, converge in distribution to a Poisson point process. By passing to the corresponding Palm distributions, we show that the rescaled slope gaps converge jointly to independent exponential random variables, which is in stark contrast to the known non-exponential gap distributions in a fixed genus.
\end{abstract}

\maketitle

\section{Introduction}

\subsection{Translation surfaces and saddle connections}
A translation surface is a compact, connected Riemann surface $X$ equipped with a non-zero holomorphic one-form $\omega$. It can be represented by finitely many Euclidean polygons whose sides are identified in pairs by translations. The one-form $\omega$ determines a flat metric $|\omega|^2$ on $X$, with the zeros of $\omega$ as its conical singularities. A \emph{saddle connection} is a geodesic segment joining two zeros of $\omega$ with no zeros in its interior. Integrating $\omega$ over a saddle connection $\gamma$ defines the \emph{holonomy vector},
\[
\int_\gamma\omega\in\mathbb C,
\]
whose argument in $[0,2\pi)$ defines the direction of $\gamma$. One classical motivation for studying translation surfaces and saddle connections comes from the dynamics of polygonal billiards.

The counting and distribution of saddle connections and their directions have a long and rich history due to the interplay of the dynamics and geometry in the moduli space of translation surfaces. For the counting problem, much work has investigated the quadratic growth of the number of saddle connections with bounded length; see, for example, \cite{Ma1,Ma2,Veech1998,EskinMasur2001,EskinMirzakhaniMohammadi2015}.
In this paper, we focus on the distribution of their directions.

Masur~\cite{Mas86} showed that the set of saddle connection directions is dense in the circle for every translation surface. Vorobets~\cite{Vorobets1996} and Eskin--Masur~\cite{EskinMasur2001} established equidistribution of saddle connection directions for almost every surface in a stratum. Dozier~\cite{Dozier2019} showed that, for every translation surface, weak limits of the distribution of saddle connection directions are in the Lebesgue measure class on the circle. The fine-scale statistics of saddle connection directions were studied by Athreya--Chaika~\cite{AthreyaChaika2012}. We recall their results and subsequent developments in the next subsection.

\subsection{Gap distributions}

Given $g \ge 1$, let $\kappa$ be a tuple of positive integers summing to $2g-2$, and let $\mathcal{H}(\kappa)$ denote the stratum of translation surfaces $(X,\omega)$ such that the multiplicities of the zeros of $\omega$ are given by $\kappa$. The group $\mathrm{GL}^+(2,\mathbb R)$ acts on $\mathcal H(\kappa)$ through its linear action on polygonal representations of translation surfaces. Let $\mathcal{H}^1(\kappa)\subset \mathcal{H}(\kappa)$ denote the locus of unit-area surfaces. Then $\mathcal{H}^1(\kappa)$ is invariant under the action of $\mathrm{SL}(2,\mathbb R)$.

Masur and Veech~\cite{Masur1982,Veech1982} independently showed each connected component of $\mathcal{H}^1(\kappa)$ carries a canonical, finite $\mathrm{SL}(2,\mathbb R)$-invariant measure, now called the \emph{Masur--Veech measure}. Fix a connected component $\mathcal{C}$ of $\mathcal{H}^1(\kappa)$, and denote by $\mu_{\mathcal C}$ the Masur--Veech measure on $\mathcal C$, normalized so that $\mu_{\mathcal C}(\mathcal C)=1$.

For $\omega\in\mathcal C$ and $R>0$, let $\Theta^R_\omega$ be the set of directions of saddle connections on $\omega$ of length at most $R$. Let $D_R(\omega)=\#\Theta^R_\omega$ and write
\[
    \Theta^R_\omega
    =
    \left\{
    0\le \theta_1<\cdots<\theta_{D_R(\omega)}<2\pi
    \right\}.
\]
Define the gaps by
\begin{equation}\label{equ:original-gap-sequence}
    \Delta_i^\omega
    :=
    \theta_{i+1}-\theta_i,
    \qquad
    1\le i\le D_R(\omega),
\end{equation}
where $\theta_{D_R(\omega)+1}=\theta_1+2\pi$.

Athreya and Chaika~\cite{AthreyaChaika2012} proved that for $\mu_{\mathcal C}$-almost every $\omega$, the limiting gap distribution exists. More precisely, there is a distribution function $G_{\mathcal C}:[0,\infty)\to[0,1]$ such that
\begin{equation}\label{equ:gap-distribution-existence}
    \lim_{R\to\infty}
    \frac{
    \#\left\{
    1\le i\le D_R(\omega):
    R^2\Delta_i^\omega\le t
    \right\}}
    {D_R(\omega)}
    =
    G_{\mathcal C}(t).
\end{equation}

Explicit formulas and stronger regularity results for specific lattice surfaces are proven in \cite{Al25,AthreyaChaikaLelievre2015,AthreyaCheung2014,BMMUW2023,MY25,UyanikWork2016}. Piecewise real analyticity of the slope gap distribution of any lattice surface is proven in \cite{KumanduriSanchezWang2024}. For non-lattice surfaces, an explicit slope gap distribution is computed in \cite{Sanchez2022} for identical tori glued along a slit and piecewise real analyticity of slope gap distribution for $\mathcal H^1(2)$ follows from \cite{Work2020}. Other related works include \cite{BS26} that proves the slope gap distribution is not exponential in genus 2, potentially outside of torus covers, \cite{OsmanSoutherlandWang2025} that proves an effective convergence to the gap distribution in the lattice case, and \cite{ChaikaFairchild2024} who considered rates for gaps about the horizontal direction of generic surfaces.

Despite these developments, no general regularity results or quantitative descriptions appear to be known for the gap distributions $G_{\mathcal C}$ in arbitrary genera.

In this paper, we approach these questions from the point process viewpoint.

\subsection{Point process approach}

We identify a locally finite subset $\zeta\subset\mathbb R$ with the \emph{counting measure} $\sum_{x\in\zeta}\delta_x$. Let $\mathcal N_s(\mathbb R)$ denote the space of locally finite simple counting measures on $\mathbb R$. Here, simple means the mass of $\zeta$ at every $x\in\mathbb R$ is at most one.  We regard random elements of $\mathcal N_s(\mathbb R)$ as point processes; the relevant definitions are recalled in Section~\ref{subsec:point-process}.

For every $\omega\in\mathcal C$ and $R>0$, extend the directions in $\Theta_\omega^R$ $2\pi$-periodically to $\mathbb R$ and denote the resulting subset by $\widetilde \Theta_\omega^R$. Choosing $\theta$ uniformly in $[0,2\pi)$, define the \emph{direction point process} $\xi_\omega^R$ by
\[
\xi_\omega^R(\theta):= 
\widetilde{\Theta}^R_\omega - \theta 
= 
\left\{\varphi-\theta: \varphi\in\widetilde\Theta_\omega^R\right\}.
\]
To describe the limiting distribution of this family as $R\to\infty$, let $\Lambda_\omega$ denote the set of holonomy vectors of saddle connections on $\omega$, and set $V:=(0,1)\times\mathbb R$. For $\omega$ distributed according to $\mu_{\mathcal C}$, define the \emph{slope point process} $\xi_{\mathcal C}$ by
\[
\xi_{\mathcal C}(\omega) := \left\{ \frac{v_y}{v_x}: v=(v_x,v_y)\in\Lambda_\omega\cap V \right\}.
\]
The direction and slope point processes are studied in detail in Section~\ref{subsec:direction-and-slope-pp}.

For $a\in\mathbb R\setminus\{0\}$ and $\zeta\in\mathcal N_s(\mathbb R)$, we write $a\zeta$ corresponding to the subset $\{ax:x\in\zeta\}$. We have the following convergence of rescaled point processes.

\begin{theorem}\label{thm:direction-point-process}
For every connected component $\mathcal C$, for $\mu_{\mathcal C}$-almost every $\omega\in\mathcal C$,
\[
    R^2\xi^R_\omega
    \xrightarrow{d}
    \xi_{\mathcal C}
    \qquad\text{as }R\to\infty.
\]
\end{theorem}

Here, $\xrightarrow{d}$ denotes convergence in distribution on $\mathcal N_s(\mathbb R)$; see Section~\ref{subsec:point-process} for the relevant definitions. The scaling by $R^2$ reflects the quadratic growth of $\#D_R(\omega)$ in $R$.

A standard probabilistic tool for studying gaps in point processes is the Palm distribution; see Section~\ref{subsec:Palm-distribution} for an introduction. We consider the convergence of the Palm distributions of $\xi_\omega^R$ to that of $\xi_{\mathcal C}$. This identifies the limiting gap distributions, and more generally the joint gap distributions, in terms of the Palm distribution of $\xi_{\mathcal C}$. This identification also allows us to study their regularity.

\begin{theorem}\label{thm:joint-gap-distributions}
Let $\mathcal C$ be a connected component of $\mathcal H^1(\kappa)$. For every $m\in\mathbb N$ and any distinct $j_1,\ldots,j_m\in\mathbb Z$, there exists a function
\[
    G_{\mathcal C;\,j_1,\ldots,j_m}
    :
    [0,\infty)^m\to[0,1]
\]
satisfying the following.
\begin{enumerate}
    \item
    For $\mu_{\mathcal C}$-almost every $\omega\in\mathcal C$ and every
    $t_1,\ldots,t_m\ge0$,
    \begin{equation}\label{equ:joint-gap-distributions-existence}
    \begin{aligned}
    &\lim_{R\to\infty}
    \frac{1}{D_R(\omega)}
    \#\left\{
    1\le i\le D_R(\omega):
    R^2\Delta^\omega_{i+j_\ell}\le t_\ell
    \text{ for }1\le\ell\le m
    \right\}
    \\
    &\hspace{4cm}
    =
    G_{\mathcal C;\,j_1,\ldots,j_m}
    (t_1,\ldots,t_m).
    \end{aligned}
    \end{equation}
    Here the indices of $\Delta_i^\omega$ are understood modulo $D_R(\omega)$.

    \item The function $G_{\mathcal C;\,j_1,\ldots,j_m}$ is continuous on $[0,\infty)^m$, strictly positive on $(0,\infty)^m$, and for every $a>0$ and $ (t_\ell),(t_\ell')\in[a,\infty)^m$, 
    \[
        \big| G_{\mathcal{C};\, j_1,\ldots,j_m}(t_1',\ldots,t_m') 
        - 
        G_{\mathcal{C};\, j_1,\ldots,j_m}(t_1,\ldots,t_m) \big| 
        \le 
        2e
        \left(2+\frac{1}{a}\right) \sum_{1\le \ell\le m} |t_\ell'-t_\ell|.
    \]

    \item
    The function $G_{\mathcal C;\,j_1,\ldots,j_m}$ is absolutely continuous for $m\le2$, but is not absolutely continuous for $m\ge\dim_{\mathbb R}\mathcal C$.
\end{enumerate}
\end{theorem}

The Palm distribution identification mentioned above takes the following form. Let $P_{\mathcal C}$ be the law of the slope point process $\xi_\mathcal{C}$ on $\mathcal N_s(\mathbb R)$. The Palm distribution associated to $P_{\mathcal C}$, denoted by $P_{\mathcal C}^0$, is a probability measure on the space of bi-infinite simple counting measures on $\mathbb R$ with a point at zero. For such a counting measure $\zeta$, write its points as
\[
\cdots<x_{-1}(\zeta)<x_0(\zeta)= 0<x_1(\zeta)<\cdots.
\]
Then
\[
G_{\mathcal C;\,j_1,\ldots,j_m}(t_1,\ldots,t_m)
=
P_{\mathcal C}^0
\bigl(
x_{j_\ell+1}(\zeta) - x_{j_\ell}(\zeta)
\le 
t_\ell
\text{ for }1\le\ell\le m
\bigr).
\]

It remains to understand whether the Lipschitz regularity in (2) extends to the boundary of $[0,\infty)^m$. Such an extension appears to require an asymptotic estimate for the number of pairs of saddle connections as the area spanned by their holonomy vectors tends to zero. Athreya, Fairchild, and Masur~\cite{AthreyaFairchildMasur2023} studied a related counting problem with a fixed upper bound on this area.

In~(3), absolute continuity means that $G_{\mathcal C;,j_1,\ldots,j_m}$ admits an $L^1$ density on $[0,\infty)^m$. The absolute continuity of $G_{\mathcal C;,j_1,\ldots,j_m}$ remains open when 
\[ 
3\le m\le\dim_{\mathbb R}\mathcal C-1. 
\]

We also discuss how the Palm distribution formulation is related to~\cite{AthreyaChaika2012} and to the Poincar\'e section for the horocycle flow; see Remark~\ref{rmk:connection-AC} and Section~\ref{subsec:connection-to-Poincare-section}.

\subsection{General ergodic measures}

The preceding results hold in greater generality. Let $\mu$ be an ergodic $\mathrm{SL}(2,\mathbb R)$-invariant probability measure on $\mathcal H^1(\kappa)$ and let $\mathcal M:=\operatorname{supp}\mu$. By the breakthrough work of Eskin--Mirzakhani~\cite{EskinMirzakhani2018} and Eskin--Mirzakhani--Mohammadi~\cite{EskinMirzakhaniMohammadi2015}, the support $\mathcal{M}$ is an immersed affine invariant submanifold. We recall the relevant definitions in Section~\ref{subsec:AIS}.

We consider the following assumption:
\begin{enumerate}[label=($\ast$), ref=($\ast$)]
    \item\label{ass:circle-limit} $\mu$-generic points have circle limit measure $\mu$. That is, if $\omega$ is $\mu$-generic, then for each bounded, measurable $f$, we have
    $$\lim_{t\to\infty}\frac{1}{2\pi}\int_0 ^{2\pi}f(g_{-t}r_{\theta}\omega)\, d\theta= \int f \,d\mu$$
    where $g_t = \begin{pmatrix}
e^{t/2} & 0 \\ 0 & e^{-t/2} \\
\end{pmatrix} $ and $r_{\theta}=\begin{pmatrix}
\cos(\theta) & -\sin(\theta) \\ \sin(\theta) & \cos(\theta) \\
\end{pmatrix} $.
\end{enumerate}

The assumption~\ref{ass:circle-limit} is proved for $\mu_\mathcal C$-almost every surface with respect to the Masur-Veech measure by Eskin--Masur \cite{EskinMasur2001}. When $\mathcal M$ is the orbit of a Veech surface, then assumption~\ref{ass:circle-limit} follows from Proposition 2.2.1 of Kleinbock-Margulis \cite{MR1359098}. Other special cases where  assumption~\ref{ass:circle-limit} is satisfied include branched covers of Veech surfaces (Eskin--Marklof--Morris \cite{EMM06}) and the eigenform loci in genus 2 (Bainbridge--Smillie--Weiss \cite{BSW}). The assumption~\ref{ass:circle-limit} has been announced for general ergodic $\mathrm{SL}(2,\mathbb R)$-invariant probability measures by Solan.

We prove the following more general version of Theorem~\ref{thm:joint-gap-distributions}; a complete statement is given in Theorem~\ref{thm:general-joint-gap-distributions}.

\begin{theorem}\label{thm:general-joint-gap-distributions-in-intro}
    Assume that $\mu$ satisfies~\ref{ass:circle-limit}. Then, for every $m\in\mathbb N$ and any distinct $j_1,\ldots,j_m\in\mathbb Z$, there exists a distribution function
    \[
    G_{\mathcal{M};\,j_1,\ldots,j_m}:[0,\infty)^m\to[0,1]
    \]
    for which the conclusions of Theorem~\ref{thm:joint-gap-distributions} hold with $(\mathcal{C},\mu_\mathcal{C})$ replaced by $(\mathcal{M},\mu)$ and $G_{\mathcal C;\,j_1,\ldots,j_m}$ replaced by $G_{\mathcal{M};\,j_1,\ldots,j_m}$, except that the strict positivity assertion in~(2) is omitted.
\end{theorem}

When $\mathcal M$ is the orbit of a Veech surface, the conclusions can be sharpened further: the joint gap distribution functions are globally Lipschitz on $[0,\infty)^m$, and the absolute continuity question is determined for every $m\in\mathbb N$; see Remark~\ref{rmk:Veech-gap-distributions}. On the other hand, it follows from~\cite[Lemma~3.6]{AthreyaChaika2012} that $G_{\mathcal{M};\,j_1,\ldots,j_m}$ vanishes in a neighborhood of the origin. Hence, the strict positivity assertion cannot be extended to arbitrary ergodic measures.

\subsection{Large genus asymptotics}

We next study how these point processes and gap distributions behave as the genus tends to infinity. 

Our study is motivated by two related developments. One is that recent work has begun to uncover the geometry of random translation surfaces in large genus. 
Explicit upper bounds for the expected covering radius of translation surfaces in terms of the genus are given in~\cite{MasurRafiRandecker2022}.
Delecroix--Goujard--Zograf--Zorich have a series of works related to the large genus geometry of square-tiled surfaces and associated counting problems; see, for example, \cite{DelecroixGoujardZografZorich2021,DelecroixGoujardZografZorich2022,DelecroixGoujardZografZorich2026}.
Masur, Rafi, and Randecker~\cite{masur2025lengthssaddleconnectionsrandom} studied the limiting distribution of the number of short saddle connections on random translation surfaces. Bowen, Rafi, and Vallejos~\cite{BRV} announced a Benjamini--Schramm limit for translation surfaces as genus tends to infinity. Forni, Rafi, and Zorich recently announced ongoing work on the geometry of random quadratic differentials in large genus and in genus zero with a large number of poles.

Another is that Athreya~\cite{Athreya2016Gap} raised the question of comparing gap distributions for saddle connection directions with those of independent random directions. The normalized gaps for independent uniformly distributed directions have an exponential limiting distribution. The explicitly computed gap distributions for translation surfaces are non-exponential; see \cite{Al25,AthreyaChaikaLelievre2015,AthreyaCheung2014,BMMUW2023,MY25,UyanikWork2016} for some examples in the lattice case and \cite{Sanchez2022} for a non-lattice example. See also \cite{BS26} where they prove that the gap distribution of a generic genus 2 surface is non-exponential by analyzing the tail.

Our large genus asymptotics show that rescaled saddle connection slope gaps become asymptotically independent and exponentially distributed as the genus tends to infinity. For each $g\ge 2$, let $\mathcal H_g:=\mathcal H_g^1(1^{2g-2})$ be the area-one locus of the principal stratum in genus $g$. Write
\[
\xi_g:=\xi_{\mathcal H_g},\qquad G_{g;\,j_1,\ldots,j_m}:=G_{\mathcal H_g;\,j_1,\ldots,j_m}
\]
for the slope point process and its finite-dimensional gap distribution.

\begin{theorem}\label{thm:large-genus-asymptotics}
As $g\to\infty$,
\[
    g^2\xi_g\xrightarrow{d}\Pi_8,
\]
where $\Pi_8$ denotes the Poisson point process on $\mathbb R$ with intensity $8$. Moreover, for every $m\in\mathbb N$ and every collection of distinct indices $j_1,\ldots,j_m\in\mathbb Z$,
\[
    \lim_{g\to\infty}
    \sup_{t_1,\ldots,t_m\ge0}
    \left|
    G_{g;\,j_1,\ldots,j_m}(t_1,\ldots,t_m)
    -
    \prod_{\ell=1}^m
    \left(1-e^{-8g^2t_\ell}\right)
    \right|
    =
    0.
\]
\end{theorem}

The scaling by $g^2$ is chosen so that the intensity of the rescaled slope point process converges to a finite nonzero limit. Here, the intensity is the expected number of points in a unit interval.

Thus, although in all known cases of fixed genus the exponentially distributed behavior of gaps has not been observed, it emerges as $g\to\infty$.

The influential work of Mirzakhani--Petri \cite{MirzakhaniPetri2019} proved that the length spectrum of a random hyperbolic surface (sampled by Weil-Petersson volume) converges to a Poisson point process as genus tends to infinity.  Later Masur--Rafi--Randecker \cite{masur2025lengthssaddleconnectionsrandom}, whose work inspired this line of inquiry, considered a similar question for random translation surfaces of large genus. In combination with Theorem~\ref{thm:large-genus-asymptotics}, these works indicate that various geometric quantities such as lengths or slopes of saddle connections de-correlate as the genus tends to infinity.

The passage from the point process limit to the joint gap asymptotics uses the corresponding convergence of Palm distributions. The same strategy also gives point process and joint gap limits for the other families of strata considered by~\cite{masur2025lengthssaddleconnectionsrandom}; see Section~\ref{subsec:other-strata-pp} and Section~\ref{subsec:other-strata-gap}.

\subsection{Related Literature}

Our work uses the rich geometry of the moduli space of translation surfaces and point processes to obtain information about the fine-scale behavior of directions and slopes of saddle connections on these surfaces. The study of using point processes in geometric contexts such as hyperbolic surfaces or the space of lattices is vast and deep. We mention some other results in the literature that use point processes to obtain geometric and number theoretic information.  

 Budzinski--Curien--Petri \cite{BudzinskiCurienPetri2025} consider a Poisson point process on hyperbolic surfaces to partition the surface into Voronoi regions to prove there is a gap from 1 for the  Cheeger constant of large genus hyperbolic surfaces. This shows there are limits to using Cheeger's inequality to prove that a closed hyperbolic surface of large genus has an optimal spectral gap.
 
Marklof--Welsh \cite{MarklofWelsh2023} use a line process on the hyperbolic plane to establish limit laws for the distribution in small intervals of the roots of certain quadratic congruence equations. Marklof \cite{Marklof2026} uses planar point processes whose directions encode the fractional parts of $\sqrt{n}$. The gap distribution of the fractional parts of $\sqrt{n}$ was investigated in Elkies--McMullen \cite{ElkiesMcMullen2004} using unipotent flows on the space of affine lattices. Marklof--Str{\"o}mbergsson \cite{MarklofStrombergsson2010} also used homogeneous dynamics to study the distribution of visible affine lattice points.
  
$\textrm{SL}(2,\mathbb R)$- invariant point processes are used in Athreya--Ghosh \cite{AthreyaGhosh2018} to understand probabilistic Diophantine approximation questions. An axiomatic framework is developed in Athreya \cite{Athreya2024} using $\textrm{SL}(2,\mathbb R)$-invariant point processes that recovers Athreya--Ghosh \cite{AthreyaGhosh2018} and applies to various counting problems in homogeneous dynamics and flat surfaces. This is in the spirit of Proposition~\ref{prop:joint-count-convergence}.

\subsection{Proof strategy}

We work more generally with an ergodic $\mathrm{SL}(2,\mathbb R)$-invariant probability measure $\mu$ on $\mathcal H^1(\kappa)$. Assuming that $\mu$ satisfies~\ref{ass:circle-limit}, we identify an underlying point process convergence behind Theorem~3.2 of Athreya--Chaika \cite{AthreyaChaika2012}; see Section~\ref{subsec:fixed-stratum-convergence}. This viewpoint makes it possible to leverage point process theory in this context.

To pass from the slope point process to gap statistics, we introduce the Palm distribution, a standard tool in point process theory. We first reformulate the convergence~\eqref{equ:gap-distribution-existence} in terms of Palm distributions. We obtain convergence of the Palm distribution by utilizing geometry and dynamics on the moduli space of translation surfaces. In particular, we use circle averages to obtain information about the finite dimensional distributions of the point process and Siegel--Veech theory for the intensities. Then a general convergence theorem for Palm distributions transfers the point process convergence to the convergence of the corresponding finite-dimensional gap distributions, giving Theorem~\ref{thm:joint-gap-distributions}(1). 
As a result, these gap distributions can be expressed as expectations of certain saddle connection counting functions over the stratum; see~\eqref{equ:rephrase-slope-gap-distribution}. 

We use this expression to study the regularity of gap distributions. For the Lipschitz estimates, we rewrite the expression as the expectation of an averaged counting function along the Teichm\"uller geodesic flow. A careful estimate of this averaged counting function then gives the desired Lipschitz bound.
For absolute continuity in dimension $m=2$, we work in local period coordinates and consider the maps defined by differences of slopes of saddle connections. The submersion properties of these maps yield absolute continuity. For the failure of absolute continuity in sufficiently large dimensions, we use the slope difference maps in the opposite direction: their images are contained in subsets of Lebesgue measure zero.
These prove Theorem~\ref{thm:joint-gap-distributions}(2)--(3).

We follow a similar approach for the large genus asymptotics. We first analyze the underlying convergence of the slope point processes. Using the Teichm\"uller flow, we renormalize the saddle connections contributing to the slope point processes to the length scale $1/g$. This brings us to the setting studied by Masur, Rafi, and Randecker~\cite{masur2025lengthssaddleconnectionsrandom}, who studied the limiting distributions of lengths of saddle connections at this scale as $g\to\infty$. We generalize their argument to obtain a rescaled convergence of point processes associated with holonomy vectors; see Theorem~\ref{thm:holonomy-vectors-ppp}. Together with the large genus asymptotics of the relevant Siegel--Veech constants, we can then apply Palm distributions once more to obtain convergence of the finite-dimensional gap distributions, yielding Theorem~\ref{thm:large-genus-asymptotics}.

\subsection*{Disclosure of AI use}
The first author used ChatGPT (GPT-5.6 Sol, OpenAI) for mathematical discussions. 
It suggested proof strategies and helped develop the arguments for the regularity of gap distributions in Proposition~\ref{prop:one-dimensional-Lipschitz} and Proposition~\ref{prop:absolutely-continuous}.
It also assisted in adapting an argument from~\cite{AthreyaChaika2012} to prove Lemma~\ref{lem:positivity-gap-distributions} and in the derivation in Remark~\ref{rmk:connection-AC}. 

ChatGPT was also used for literature searches in point process theory and for language editing. In particular, the first author learned of Palm distribution theory through these searches. Its applications in this paper were developed by the authors. All AI-identified references were checked against the original sources before citation.

The main research questions and the overall approach of the paper were developed by the authors. The authors completed all proofs, wrote the manuscript, and take full responsibility for its content.

\subsection*{Acknowledgments}

This work began at the Clay Mathematics Institute and the Heilbronn Institute for Mathematical Research (CMI-HIMR) summer school on ``Random Geometries and Random Matrices" at the University of Bristol. We thank CMI-HIMR, the program organizers, and the University of Bristol for  their hospitality and for providing ideal working conditions. The authors thank Anja Randecker, Ivan Yakovlev, and Anton Zorich for helpful discussions at an early stage of this project and Jayadev Athreya for comments on the paper.
 
\section{Preliminaries}

\subsection{Translation surfaces and their strata}\label{subsec:translation-surfaces}

For general background on translation surfaces and Teichm\"uller dynamics, see, for example,~\cite{Zorich2006,ForniMatheus2014,Wright2015,AthreyaMasur2024}.

A \emph{translation surface} is a triple $(X,z,\omega)$, where $X$ is a compact Riemann surface, $\omega$ is a nonzero Abelian differential, and $z=(z_1,\ldots,z_n)$ is the ordered set of its zeros. Equivalently, it is an oriented closed surface obtained from finitely many polygons by pairwise identifications of edges by translations.

Away from $z$, the metric $|\omega|^2$ is a flat Riemannian metric. Integrating $\omega$ locally gives \emph{translation charts} to $\mathbb C$, in which $\omega=dw$ and the transition maps are translations. If $z_i$ is a zero of order $k$, then locally $\omega=w^k\,dw$, so $z_i$ is a cone point of angle $2(k+1)\pi$. The area is computed by $\frac{i}{2}\int_X\omega\wedge\overline{\omega}$. We often denote a translation surface $(X,z,\omega)$ simply by $\omega$.

A \emph{saddle connection} $\gamma$ on $\omega$ is the image of a path $p:[a,b]\to X$ such that $p(a),p(b)\in z$ and $p|_{(a,b)}$ is a geodesic in $X\setminus z$ with respect to $|\omega|^2$. Each saddle connection has two orientations, and a choice of parameterization determines an orientation. The \emph{holonomy vector} of an oriented saddle connection $\gamma$ is $\int_\gamma\omega\in\mathbb C$, where the integral is computed using any parameterization inducing the given orientation.

Let $\Lambda_\omega\subset\mathbb C$ denote the set of holonomy vectors of all oriented saddle connections on $\omega$. Then $\Lambda_\omega$ is discrete and $\Lambda_\omega=-\Lambda_\omega$.

Given $g\ge 1$ and $n\ge0$, let $\Omega_{g,n}$ be the moduli space of equivalence classes of genus $g$ translation surfaces with $n$ labeled points. Two translation surfaces $(X,z,\omega)$ and $(X',z',\omega')$ are equivalent if there exists a biholomorphic map $f:X\to X'$ such that $f(z_i) = z'_i$ for each $i$ and $f^\ast\omega' = \omega$.

For $(X,z,\omega)\in\Omega_{g,n}$, let $\kappa=(k_1,\ldots,k_n)\in\mathbb Z_{>0}^n$, where $k_i$ is the order of $\omega$ at $z_i$. Since $\sum_{i=1}^nk_i = 2g-2$, we have
\[
\Omega_{g,n} = \bigsqcup_\kappa \mathcal{H}(\kappa),
\]
where $\kappa$ ranges over all $n$-tuples of positive integers summing to $2g-2$. The \emph{stratum} $\mathcal{H}(\kappa)$ consists of translation surfaces whose zero $z_i$ has order $k_i$ for each $i$. We denote its area-one locus by $\mathcal{H}^1(\kappa)$.

The \emph{principal stratum} is the stratum $\mathcal H(1^{2g-2})$, where $1^{2g-2}$ denotes the tuple $(1,\ldots,1)$ with $2g-2$ entries.

We recall period coordinates on $\mathcal H(\kappa)$. Let $(S,\Sigma)$ be a closed oriented surface of genus $g$ with $n$ labeled marked points. For a sufficiently small neighborhood $U$ of $\omega_0\in\mathcal H(\kappa)$, trivialize the relative homology local system over $U$ with fiber $H_1(S,\Sigma;\mathbb Z)$. Define the \emph{period map}
\[
    \Phi:U\longrightarrow H^1(S,\Sigma;\mathbb C),
    \qquad
    \Phi(\omega)(\alpha)=\int_\alpha\omega,
\]
for $\alpha\in H_1(S,\Sigma;\mathbb Z)$. After shrinking $U$ if necessary, $\Phi$ is a coordinate chart, called a \emph{period coordinate chart}. In particular, $\dim_{\mathbb C}\mathcal H(\kappa)=2g+n-1$.

Since $\mathcal H(\kappa)$ is in general an orbifold, we replace $U$, whenever necessary, by a local finite cover on which the relative homology local system is trivial.

\subsection{\texorpdfstring{$\mathrm{SL}(2,\mathbb R)$}{SL(2,R)}-action and affine invariant submanifolds}\label{subsec:AIS}

The group $\mathrm{GL}^+(2,\mathbb R)$ acts on each stratum $\mathcal H(\kappa)$: given $A\in\mathrm{GL}^+(2,\mathbb R)$ and $\omega\in\mathcal H(\kappa)$, define $A\cdot\omega$ to be the translation surface obtained by post-composing the translation charts on $X$ with the linear map $A$. Equivalently, if $\omega$ is represented as a collection of polygons with edge identifications, then $A\cdot \omega$ is obtained by applying the linear map $A$ to each polygon while maintaining the same edge identifications.

The area-one locus $\mathcal H^1(\kappa)$ is invariant under the $\mathrm{SL}(2,\mathbb R)$-action. Masur and Veech~\cite{Masur1982,Veech1982} independently showed that each $\mathcal H^1(\kappa)$ carries a canonical finite $\mathrm{SL}(2,\mathbb R)$-invariant measure, now known as \emph{Masur--Veech measure}.

A stratum may have several connected components, classified by Kontsevich and Zorich~\cite{KontsevichZorich2003}. Throughout, we use $\mathcal C$ to denote a connected component of $\mathcal H^1(\kappa)$, and let $\mu_{\mathcal C}$ be the restriction of the Masur--Veech measure to $\mathcal C$, normalized to be a probability measure.

The $\mathrm{SL}(2,\mathbb R)$-action is ergodic with respect to the Masur--Veech measure on each connected component of $\mathcal{H}^1(\kappa)$. Two subgroups that will play an important role are the \emph{geodesic} and \emph{horocyclic} subgroups 
\[
A = \left\{g_t = \begin{pmatrix}
e^{t/2} & 0 \\ 0 & e^{-t/2} \\
\end{pmatrix} : t \in \mathbb R\right\} 
\qquad\text{and}\qquad
N = \left\{h_s = \begin{pmatrix}
1 & 0 \\ -s & 1 \\
\end{pmatrix} : s \in \mathbb R\right\}\text.
\]
The corresponding actions are called the \emph{geodesic flow} and the \emph{horocycle flow}, respectively. We will refer to them simply as $(g_t)$ and $(h_s)$.

We will work in the more general setting of ergodic $\mathrm{SL}(2,\mathbb R)$-invariant probability measures. This leads to the notion of an affine invariant submanifold. A subset $\mathcal M\subset\mathcal H^1(\kappa)$ is called an
\emph{affine invariant submanifold} if
\[
    \mathbb R\mathcal M
    :=
    \{r\omega:r\in\mathbb R\setminus\{0\},\ \omega\in\mathcal M\}
    \subset\mathcal H(\kappa)
\]
is a properly immersed connected suborbifold such that, in every sufficiently small period coordinate chart, its image is a finite union of linear subspaces defined by homogeneous linear equations with real coefficients.

By the work of Eskin--Mirzakhani~\cite{EskinMirzakhani2018} and Eskin--Mirzakhani--Mohammadi~\cite{EskinMirzakhaniMohammadi2015}, every ergodic $\mathrm{SL}(2,\mathbb R)$-invariant probability measure is an affine measure supported on an affine invariant submanifold, and every $\mathrm{SL}(2,\mathbb R)$-orbit closure is an affine invariant submanifold.

Throughout, $\mu$ denotes an ergodic $\mathrm{SL}(2,\mathbb R)$-invariant probability measure, and we write
\[
    \mathcal M := \operatorname{supp}\mu.
\]

In particular, every connected component of $\mathcal{H}^1(\kappa)$ is an affine invariant submanifold. Another example comes from Veech surfaces. A \emph{Veech surface} is a translation surface $(X,\omega)$ whose stabilizer in $\mathrm{SL}(2,\mathbb R)$ is a lattice. Its $\mathrm{SL}(2,\mathbb R)$-orbit is closed and is an affine invariant submanifold.

\subsection{Siegel--Veech theory}\label{subsec:Siegel-Veech}

We first recall the Siegel--Veech formula introduced by Veech~\cite{Veech1998}. 

Let $B_c(\mathbb R^d)$ denote the space of Borel measurable bounded compactly supported functions on $\mathbb R^d$. We denote Lebesgue measure on $\mathbb R^d$ by $\mathrm{Leb}_d$, and simply write $\mathrm{Leb}$ when $d=1$.

Let $\mu$ be an ergodic $\mathrm{SL}(2,\mathbb{R})$-invariant probability measure on $\mathcal{H}^1(\kappa)$ and let $\mathcal{M}$ be its support. For $f\in B_c(\mathbb R^2)$, the \emph{Siegel--Veech transform} of $f$ is defined by
\[
    \widehat f(\omega)
    :=
    \sum_{v\in\Lambda_\omega} f(v),
\]
for $\omega\in\mathcal{H}^1(\kappa)$. Since $\Lambda_\omega$ is discrete, the sum is finite.

\begin{theorem}[Siegel--Veech formula {\cite{Veech1998}}]\label{thm:Siegel-Veech-formula}
There exists a constant $c(\mu)$ such that
\begin{equation}\label{equ:Siegel-Veech-formula}
    \int_{\mathcal H^1(\kappa)}
    \widehat f(\omega)\,d\mu(\omega)
    =
    c(\mu) \int_{\mathbb R^2} f(v)\,d\operatorname{Leb}_2(v)
\end{equation}
for every $f\in B_c(\mathbb R^2)$.
\end{theorem}

The constant $c(\mu)$ is called the \emph{Siegel--Veech constant}. If $\mu$ is the Masur--Veech measure $\mu_\mathcal{C}$ on a connected component $\mathcal{C}$ of $\mathcal{H}^1(\kappa)$, we simply write $c(\mathcal C)$.

One motivation for Veech's introduction of the Siegel--Veech transform is to understand counting problems. Given $\omega$ and $R>0$, define
\[
N_R(\omega):=\#\left(\Lambda_\omega\cap B(0,R)\right),
\]
where $B(0,R)$ denotes the closed ball of radius $R$ centered at the origin. This should not be confused with $D_\omega(R)$ defined in the introduction, which counts the number of distinct directions of saddle connections of length at most $R$. In particular, $D_R(\omega)\le N_R(\omega)$.

Masur~\cite{Ma1,Ma2} established quadratic upper and lower bounds for $N_R(\omega)$. Eskin-Masur~\cite{EskinMasur2001} showed that for $\mu$-almost every $\omega\in\mathcal{H}^1(\kappa)$,
\[
\lim_{R\to\infty}\frac{1}{\pi R^2}N_R(\omega) = c(\mu).
\]

Eskin-Masur-Zorich~\cite{EskinMasurZorich2003} introduce a method for computing Siegel--Veech constants by relating them to volumes of strata.
The large genus behavior of Masur--Veech volumes and Siegel--Veech constants has been studied in many works; see, for example, \cite{ChenMollerZagier2018,Aggarwal2019,Aggarwal2020,ChenMollerSauvagetZagier2020,Randecker2025}.

We will use the asymptotic formula in~\cite[Equation~(24)]{Randecker2025}, but our convention differs from that in~\cite{Randecker2025}. Recall that $\Lambda_\omega=-\Lambda_\omega$, so both orientations of each holonomy vector are counted in the definition of $c(\mathcal C)$. In contrast, \cite{Randecker2025} counts saddle connections without orientation. Thus, the Siegel--Veech constant in~\cite[Equation~(24)]{Randecker2025} is $c(\mathcal C)/2$ in our notation, and the equation can be rewritten as
\begin{equation}\label{equ:Siegel-Veech-asymptotic}
    \frac{c(\mathcal C)}{2}
    =
    \frac{(2g + n -2)^2}{2}\cdot\left(1+O\left(\frac{1}{g}\right)\right).
\end{equation}

\subsection{Point processes}\label{subsec:point-process}

We recall the notions and results concerning point processes that will be used later. We largely follow the terminology and conventions of~\cite{DaleyVereJones2008,Kallenberg2017,LastPenrose2018Poisson}.

Throughout, for a locally compact second countable Hausdorff space $M$, we denote by $\mathcal{B}(M)$ its Borel $\sigma$-algebra and by $\mathcal{P}(M)$ the space of probability measures on $M$. We denote by $B_c(M)$ the space of bounded Borel measurable functions on $M$ with compact support.

A \emph{counting measure} on $M$ is a measure of the form $\sum_{i}\delta_{x_i}$, where $(x_i)$ is a finite or countably infinite sequence of points in $M$. Given an at most countable subset $Z\subset M$, we identify $Z$ with the counting measure $\sum_{x\in Z}\delta_x$.

A Borel measure $\zeta$ on $M$ is \emph{locally finite} if $\zeta(A)<\infty$ for every relatively compact Borel subset $A\subset M$. Similarly, an at most countable subset $Z\subset M$ is \emph{locally finite} if $\#Z\cap A<\infty$ for every relatively compact Borel subset $A\subset M$.

Let $\mathcal{N}(M)$ denote the set of locally finite counting measures on $M$. A counting measure $\zeta\in \mathcal N(M)$ is \emph{simple} if $\zeta(\{x\})\le 1$ for every $x\in M$. We denote by $\mathcal{N}_s(M)\subset\mathcal{N}(M)$ the set of locally finite simple counting measures. Under the above identification, every at most countable locally finite subset of $M$ determines an element of $\mathcal N_s(M)$.

We endow $\mathcal{N}(M)$ with the \emph{vague topology}. That is, $\zeta_n\to \zeta$ if and only if 
\[
\int_{M} f(x)\,d\zeta_n(x)\to \int_{M}f(x)\,d\zeta(x)
\]
for every compactly supported function $f\in C_c(M)$. Here and below, $d\zeta(x)$ denotes integration with respect to the counting measure $\zeta$. Thus, if $\zeta=\sum_i\delta_{x_i}$, then
\[
    \int_{M} f(x)\,d\zeta(x)
    =
    \sum_i f(x_i).
\]

Let $(\Omega,\mathcal F,\mathbb P)$ be a probability space. A \emph{point process} on $M$ is a random variable
\[
    \xi:(\Omega,\mathcal F)
    \longrightarrow
    \left(\mathcal N(M),
    \mathcal{B}(\mathcal N(M))\right).
\]
Its distribution is the probability measure
\[
    P_\xi:=\xi_\ast\mathbb P
    \in\mathcal P\left(\mathcal N(M)\right).
\]
Thus, for every $A\in \mathcal{B}(\mathcal N(M))$, $P_\xi(A)=\mathbb P(\xi^{-1}A)$. We write $E_\xi$ for expectation with respect to $P_\xi$.

A point process $\xi$ is called \emph{simple} if $P_\xi\left(\mathcal N_s(M)\right)=1$. 

When $M=\mathbb R^d$, for $x\in\mathbb R^d$, define
\begin{equation}\label{equ:translate-notation}
    (\tau_x\zeta)(B)
    :=
    \zeta\left(\tau_x(B)\right)
    =
    \zeta(B+x),
    \qquad
    B\in\mathcal B(\mathbb R^d).
\end{equation}
In particular, $\tau_x\delta_x=\delta_0$. 

A point process $\xi$ is said to be \emph{stationary} if $(\tau_x)_\ast P_\xi = P_\xi$ for every $x\in\mathbb R^d$. Equivalently, for every $x\in\mathbb R^d$ and $A\in \mathcal{B}(\mathcal{N}(\mathbb R^d))$,
\[
    P_\xi(\tau_x\zeta\in A) = P_\xi(A).
\]

Assume that $\xi$ is a stationary point process on $\mathbb R^d$ and $E_\xi\left[\zeta([0,1]^d)\right]<\infty$. Then, by~\cite[Proposition~8.2]{LastPenrose2018Poisson}, there exists $\lambda<\infty$ such that
\[
E_\xi[B] = \lambda\,\mathrm{Leb}_d(B)
\]
for $B\in\mathcal{B}\left(\mathbb R^d\right)$. We call $\lambda$ the \emph{intensity} of $\xi$. In particular, for every $L>0$,
\begin{equation}\label{equ:intensity-of-stationary-pp}
    \lambda = \frac{E_\xi\left[\zeta\left([0,L]^d\right)\right]}{L^d}.
\end{equation}

We next recall the definition of a Poisson point process. Let $M$ be a locally compact Borel subset of $\mathbb R^d$. For $\lambda>0$, a point process $\xi$ on $M$ is called a \emph{Poisson point process} with intensity $\lambda$ if, for every collection of pairwise disjoint relatively compact Borel sets $B_1,\ldots,B_m\subset M$, the random variables
\[ 
    \xi(B_1),\ldots,\xi(B_m) 
\] 
are independent and Poisson distributed with means $\lambda\,\mathrm{Leb}_d(B_1),\ldots, \lambda\,\mathrm{Leb}_d(B_m)$. We denote this process by $\Pi_\lambda^M$. Its intensity measure is $\lambda\,\mathrm{Leb}_d|_M$.

We next introduce convergence of point processes. Let $(\xi_n)_{n\in\mathbb N}$ and $\xi$ be point processes on $M$. We say that $\xi_n$ \emph{converges in distribution} to $\xi$ as $n\to\infty$, and write $\xi_n\xrightarrow{d}\xi$
if
\[
    E_{\xi_n}\left[F(\zeta)\right]
    \longrightarrow
    E_\xi\left[F(\zeta)\right],\qquad \text{as }n\to\infty,
\]
for every bounded continuous function $F:\mathcal N(M)\to\mathbb R$, where $\mathcal N(M)$ is equipped with the vague topology.

The following characterization will be used repeatedly. A Borel set $B\subset M$ is called a $\xi$-\emph{continuity set} if $\zeta(\partial B)=0$ for $P_\xi$-almost every $\zeta\in\mathcal N(M)$.

\begin{theorem}[{\cite[Definition~11.1.IV and Theorem~11.1.VII]{DaleyVereJones2008}}]\label{thm:characterization-point-process-convergence}
Let $M$ be a locally compact second-countable Hausdorff space, and let $(\xi_n)_{n\in\mathbb N}$ and $\xi$ be point processes on $M$. Then $\xi_n\xrightarrow{d}\xi$ if and only if, for every $m\in\mathbb N$, every collection of relatively compact $\xi$-continuity sets $B_1,\ldots,B_m\in\mathcal B(M)$ and every $r_1,\ldots,r_m\in\mathbb Z_{\geq 0}$,  
\[
    P_{\xi_n}\left(
    \zeta(B_1)=r_1,\ldots,\zeta(B_m)=r_m
    \right)\longrightarrow
    P_\xi\left(
    \zeta(B_1)=r_1,\ldots,\zeta(B_m)=r_m
    \right)
\]
as $n\to\infty$.
\end{theorem}

It is enough to consider pairwise disjoint sets $B_1,\ldots,B_m$ in the theorem.

\subsection{Palm distributions and gap sequences}\label{subsec:Palm-distribution}

We first introduce Palm distributions of point processes. For our purposes, we specialize to stationary point processes on $\mathbb R$. We begin with the refined Campbell theorem, which we use to define the Palm distribution.

\begin{theorem}[{Refined Campbell Theorem~\cite[Theorem~9.1]{LastPenrose2018Poisson}}]\label{thm:refined-Campbell-theorem}
    Let $\xi$ be a stationary point process on $\mathbb R$ with intensity $\lambda\in(0,\infty)$. Then there exists a unique probability measure $P^0_\xi\in \mathcal P(\mathcal{N}(\mathbb R))$ such that
    \begin{equation}\label{equ:refined-Campbell-theorem}
        E_\xi\left[\int_{\mathbb R} f(x,\tau_x\zeta)d\zeta(x)\right] = \lambda\int_{\mathbb R}\int_{\mathcal{N}(\mathbb R)} f(x,\zeta)\,dP^0_\xi(\zeta)\,dx
        \end{equation}
    for every non-negative measurable function $f:\mathbb R\times\mathcal{N}(\mathbb R)\to[0,\infty]$.
\end{theorem}

The probability measure $P^0_\xi$ is called the \emph{Palm distribution} of $\xi$. We denote the expectation with respect to $P^0_\xi$ by $E^0_\xi$.

For $A\in\mathcal{B}(\mathcal{N}(\mathbb R))$ and $B\in\mathcal{B}(\mathbb R)$ with $0<\mathrm{Leb}(B)<\infty$, taking $f(x,\zeta):= \boldsymbol{1}_B(x)\boldsymbol{1}_A(\zeta)$ in~\eqref{equ:refined-Campbell-theorem} gives
\begin{equation}\label{equ:computation-Palm-distribution}
    \begin{aligned}
        P^0_\xi(A) &= \frac{1}{\lambda\mathrm{Leb}(B)}E_{\xi}\left[\int_{\mathbb R}\chi_B(x)\chi_A(\tau_x\zeta)d\zeta(x)\right]\\
        &= \frac{1}{\lambda\mathrm{Leb}(B)}\int_{\mathcal{N}(\mathbb{R})}\int_{\mathbb R}\chi_B(x)\chi_A(\tau_x\zeta)d\zeta(x)dP_\xi(\zeta).
    \end{aligned}
\end{equation}

We record the following convergence theorem for Palm distributions. This is a particular case of~\cite[Theorem~6.3]{Kallenberg1973}. Here, for probability measures $Q_n,Q\in\mathcal P(\mathcal N(\mathbb R))$, we say that $Q_n$ converges \emph{weakly} to $Q$ if
\[
\int_{\mathcal N(\mathbb R)} F(\zeta)dQ_n(\zeta)
\longrightarrow
\int_{\mathcal N(\mathbb R)} F(\zeta)dQ(\zeta)
\]
for every bounded continuous function $F:\mathcal N(\mathbb R)\to\mathbb R$, where $\mathcal N(\mathbb R)$ is equipped with the vague topology.

\begin{theorem}[{\cite[Theorem~6.3]{Kallenberg1973}}]\label{thm:Palm-convergence}
    Let $\xi$ and $(\xi_n)_{n\in\mathbb N}$ be stationary point processes on $\mathbb{R}$ with intensities $\lambda\in (0,\infty)$ and $(\lambda_n)_{n\in\mathbb N}$ respectively. If $\xi_n\xrightarrow{d}\xi$ and $\lambda_n\to \lambda$ as $n\to\infty$, then $P^0_{\xi_n}$ converges to $P^0_{\xi}$ weakly.
\end{theorem}

We next introduce their gap sequences. Let
\begin{equation}\label{equ:bi-infinite}
    \mathcal N_{s,\mathrm{bi}}(\mathbb R)
    :=
    \left\{
    \zeta\in\mathcal N_s(\mathbb R):
    \zeta((-\infty,0))=\zeta((0,\infty))=\infty
    \right\}
\end{equation}
be the space of bi-infinite simple counting measures on $\mathbb R$, and set
\[
    \mathcal N_{s,\mathrm{bi}}^0(\mathbb R) := \left\{ \zeta\in\mathcal N_{s,\mathrm{bi}}(\mathbb R): \zeta(\{0\})=1 \right\}.
\]
By~\cite[Section~13.3]{DaleyVereJones2008}, if $\xi$ is a stationary simple point process on $\mathbb R$ with positive finite intensity, then
\begin{equation}\label{equ:bi-infinite-concentration}
    P_\xi^0\left(\mathcal N_{s,\mathrm{bi}}^0(\mathbb R)\right)=1.
\end{equation}

For $\zeta\in\mathcal N_{s,\mathrm{bi}}(\mathbb R)$, label its points by
\[
    \zeta
    =
    \sum_{i\in\mathbb Z}\delta_{x_i(\zeta)},
    \qquad
    x_i(\zeta)<x_{i+1}(\zeta),
    \qquad
    x_0(\zeta)\leq 0<x_1(\zeta).
\]
Following~\cite[Section~2.1]{huesmann2025gapmetricsstationarypoint}, we call
\[
    \Delta\zeta
    := 
    \left(
    \Delta_i\zeta
    \right)_{i\in\mathbb Z},\qquad \text{where }\Delta_i\zeta:=x_{i+1}(\zeta)-x_i(\zeta), 
\]
the \emph{gap sequence} of $\zeta$. This defines a map
\[
    \Delta=(\Delta_i)_{i\in\mathbb Z}
    \colon
    \mathcal N_{s,\mathrm{bi}}(\mathbb R)
    \longrightarrow
    [0,\infty)^{\mathbb Z}.
\]
We equip $[0,\infty)^{\mathbb Z}$ with the product topology. The restriction $\Delta\big|_{\mathcal N_{s,\mathrm{bi}}^0(\mathbb R)}$ is continuous. Indeed, $\zeta_n\to\zeta$ converges in the vague topology of $\mathcal N_{s,\mathrm{bi}}^0(\mathbb R)$ as $n\to\infty$ if and only if the points $x_i(\zeta_n)$ converge to $x_i(\zeta)$ for each $i\in\mathbb Z$; see also
\cite[Section~2.1]{huesmann2025gapmetricsstationarypoint}. Hence, $\Delta_i\zeta_n = x_{i+1}(\zeta_n)-x_i(\zeta_n)$ converges to $\Delta_i\zeta$ for each $i$.

Then $\Delta_\ast P_\xi^0$ is a probability measure on $[0,\infty)^\mathbb{Z}$, and we call it the \emph{gap-sequence distribution} of $\xi$. 

A probability measure $P$ on $[0,\infty)^\mathbb{Z}$ is called \emph{stationary under shifts} if 
\[
P(\sigma_k(B)) = P(B),\qquad \text{for } 
B\in\mathcal{B}\left([0,\infty)^\mathbb{Z}\right),
\]
where $k\in\mathbb N$ and $\sigma_k:[0,\infty)^\mathbb Z\to[0,\infty)^\mathbb Z$ is a map given by $\sigma_k(x_i)_{i\in\mathbb{Z}} = (x_{i+k})_{i\in\mathbb Z}$. Moreover, a probability measure $P^0$ on $\mathcal N_{s,\mathrm{bi}}^0(\mathbb R)$ is \emph{stationary under shifts} if $\Delta_\ast P^0$ is stationary under shifts.

We call a point process $\xi$ \emph{nonnull} if $P_\xi(\xi=\varnothing) = 0$.

\begin{theorem}[{\cite[Theorem~13.3.I]{DaleyVereJones2008}}]\label{thm:inversion-formula}
    For $\lambda\in(0,\infty)$, the map $P_\xi\mapsto P^0_\xi$ is a bijection between the distributions of simple nonnull stationary point processes on $\mathbb R$ with intensity $\lambda$ and the probability measures $P^0$ on $\mathcal N_{s,\mathrm{bi}}^0(\mathbb R)$ that are stationary under shifts and
    \[
    \lambda \int_{\mathcal N_{s,\mathrm{bi}}^0(\mathbb R)}\Delta_0(\zeta)dP^0(\zeta)
    =
    1.
    \]
    Moreover, if $\xi$ is a simple nonnull stationary point process with intensity $\lambda$, then $P^0_\xi$ is the unique probability measure on $\mathcal N_{s,\mathrm{bi}}^0(\mathbb R)$ satisfying the Palm inversion formula:
    \begin{equation}\label{equ:Palm-inversion}
        \begin{aligned}
            E_\xi[f(\zeta)] 
            &= \lambda\,E^0_\xi\left[\int_0^{\Delta_0(\zeta)}f(\tau_x\zeta)dx\right]\\
            & = 
            \lambda
            \int_0^\infty E^0_\xi
            \left[
            f(\tau_t\zeta)
            \chi_{\{\Delta_0\zeta>t\}}(\zeta)
            \right]
            dt.
        \end{aligned}
    \end{equation}
    for any bounded measurable function $f:\mathcal N_{s}(\mathbb R)\to\mathbb R$.
\end{theorem}

For later usage, we record the distribution of gap sequence of a Poisson point process.

\begin{theorem}[{\cite[Theorem~7.2]{LastPenrose2018Poisson}}]\label{thm:gap-distributions-ppp}
    Let $\xi$ be a Poisson point process on $\mathbb R$ with intensity
    $\lambda>0$. Then
    \[
    \Delta_\ast P_\xi^0 = \bigotimes_{ j\in\mathbb Z}\mathrm{Exp}(\lambda),
    \]
    where $\operatorname{Exp}(\lambda)$ denotes the exponential distribution on $(0,\infty)$ with rate $\lambda$.
\end{theorem}

The following lemma transfers point process convergence to the corresponding gap-sequence distributions; see also~\cite[Lemma~2.10(b)]{huesmann2025gapmetricsstationarypoint}.

\begin{lemma}\label{lem:continuous-mapping-theorem}
    Let $\xi$ and $(\xi_n)_{n\in\mathbb N}$ be simple stationary point processes on $\mathbb{R}$ with intensities $\lambda\in (0,\infty)$ and $(\lambda_n)_{n\in\mathbb N}$ respectively. If $\xi_n\xrightarrow{d}\xi$ and $\lambda_n\to \lambda$ as $n\to\infty$, then $\Delta_\ast P_{\xi_n}^0\to\Delta_\ast P_{\xi}^0$ weakly.
\end{lemma}

\begin{proof}
    By Theorem~\ref{thm:Palm-convergence}, $P^0_{\xi_n}\longrightarrow P^0_\xi$ weakly as $n\to\infty$. Since $\Delta$ is continuous on $\mathcal N_{s,\mathrm{bi}}^0(\mathbb R)$, the conclusion follows from the continuous mapping theorem~\cite[Theorem~5.27]{Kallenberg2021Foundations}.
\end{proof}

For every $m\in\mathbb N$, $j_1,\ldots,j_m\in\mathbb Z$, define the finite dimensional distribution function
\begin{equation}\label{equ:finite-dimensional-distribution}
    G_{\xi;\, j_1,\ldots,j_m}(t_1,\ldots,t_m):= P_\xi^0(\Delta_{j_1}\le t_1,\ldots,\Delta_{j_m}\le t_m),
\end{equation}
for $t_1,\ldots,t_m\in\mathbb R$. Since the gaps are positive, $G_{\xi;\,j_1,\ldots,j_m}(t_1,\ldots,t_m)=0$ whenever $t_\ell\le 0$ for some $1\leq \ell\leq m$. Since the gap-sequence distribution is stationary under shifts, we have for every $m\in\mathbb N$ and $j_1,\ldots,j_m,k\in\mathbb Z$,
\begin{equation}\label{equ:interval-stationarity}
    G_{\xi;,j_1+k,\ldots,j_m+k}(t_1,\ldots,t_m)
    =
    G_{\xi;,j_1,\ldots,j_m}(t_1,\ldots,t_m),\qquad t_1,\ldots,t_m\in\mathbb R.
\end{equation}

We conclude this subsection by recording the behavior of the intensity and the distribution functions under rescaling. For $a>0$, define the scaling map $M_a:\mathcal N_s(\mathbb R)\to\mathcal N_s(\mathbb R)$ by $M_a\bigl(\sum_i\delta_{x_i}\bigr):=\sum_i\delta_{ax_i}$. For a point process $\xi$ on $\mathbb R$, define $a\xi:=M_a\circ\xi$. Then:
\begin{itemize}
    \item If $\xi$ has intensity $\lambda$, then $a\xi$ has intensity $\lambda_{a\xi} = \frac{\lambda}{a}$.
    \item $\Delta_i(a\zeta) = a\,\Delta_i\zeta$ for $i\in\mathbb Z$.
    \item The Palm distribution of $a\xi$ is the pushforward of $P_\xi^0$ under the map $\zeta\mapsto a\zeta$. Consequently, for every $m\in\mathbb N$, $j_1,\ldots,j_m\in\mathbb Z$, and $t_1,\ldots,t_m\in\mathbb R$,
    \[
        G_{a\cdot\xi;\,j_1,\ldots,j_m}(t_1,\ldots,t_m)
        =
        G_{\xi;\,j_1,\ldots,j_m}
        \left(
        \frac{t_1}{a},\ldots,\frac{t_m}{a}
        \right).
    \]
\end{itemize}

\subsection{Periodic point process}

We establish a general property of periodic point processes that will be used in the following sections. Let $\zeta\in\mathcal N(\mathbb R)$. For $L>0$, we say that $\zeta$ is $L$-periodic if $\tau_L\zeta=\zeta$.

\begin{lemma}\label{lem:periodicPP} 
Let $\zeta\in\mathcal N(\mathbb R)$ be a simple $L$-periodic counting measure for some $L>0$. Assume that $N:=\zeta([0,L))>0$. Let $\theta$ be uniformly distributed on $[0,L)$, and define the point process
\[
    \xi(\theta):=\tau_\theta\zeta .
\]
Then 
\begin{enumerate}
    \item $\xi$ is a stationary simple point process with intensity $N/L$;
    \item \textup{(Palm distribution)} For every $A\in\mathcal B(\mathcal N(\mathbb R))$,
    \[
    P_{\xi}^0(A)
    =
    \frac1N \sum_{x\in\zeta\cap[0,L)}
    \chi_A(\tau_x\zeta).
    \]
    \item \textup{(Distribution of gap sequence)} For any $j_1,\ldots,j_m\in\mathbb Z$ and $t_1,\ldots,t_m\ge 0$,
    \[
        \begin{aligned}
        &P_\xi^0\left(\Delta_{j_1}\le t_1,\ldots, \Delta_{j_m}\le t_m\right) = \frac{1}{N}{\#\left\{1\le i\le N:\Delta_{i+j_1}\zeta\le t_1,\ldots, \Delta_{i+j_m}\zeta\le t_m \right\}}.
        \end{aligned}
    \]
\end{enumerate}
\end{lemma}

\begin{proof}

(1) Let $A\in\mathcal B(\mathcal N(\mathbb R))$ and $t\in\mathbb R$. Since the function $\theta\mapsto \chi_A(\tau_\theta\zeta)$ is $L$-periodic, 
\[ 
    \begin{aligned} 
    P_\xi(\tau_t\xi\in A) &= \frac1L\int_0^L 
    \chi_A(\tau_{t+\theta}\zeta)\,d\theta= \frac1L\int_t^{t+L}
    \chi_A(\tau_\theta\zeta)\,d\theta \\
    &= \frac1L\int_0^L
    \chi_A(\tau_\theta\zeta)\,d\theta = P_\xi(\xi\in A). 
    \end{aligned} 
\] 
Thus $\xi$ is stationary.

Let $\lambda$ denote the intensity of $\xi$. Since 
\[
    \xi([0,L)) = \zeta([\theta,\theta+L)) = \zeta([0,L)) = N,
\] 
we have $\lambda = E_\xi [\xi([0,L))]/L= N/L$.

(2) Taking $B = [0,L)$ in~\eqref{equ:computation-Palm-distribution} gives
\[ 
    \begin{aligned} 
    P_\xi^0(A) 
    &= 
    \frac{1}{\lambda L} E_\xi\left[ \int_0^L \chi_A(\tau_x\xi)\,d\xi(x) \right] 
    =
    \frac1N E_\xi\left[ \int_\theta^{\theta+L} 
    \chi_A(\tau_y\zeta)\,d\zeta(y) \right]. 
    \end{aligned}
\] 
Since both $\zeta$ and $y\mapsto \tau_y\zeta$ are $L$-periodic, the last integral is independent of $\theta$. Therefore, 
\[ 
    P_\xi^0(A) 
    = \frac1N\int_0^L\chi_A(\tau_y\zeta)\,d\zeta(y) 
    = \frac1N\sum_{x\in\zeta\cap[0,L)} \chi_A(\tau_x\zeta). 
\]

(3) The points in $\zeta\cap[0,L)$ are $x_i(\zeta)$, $1\le i\le N$. Since $\Delta_j\left(\tau_{x_i(\zeta)}\zeta\right)=\Delta_{i+j}\zeta$ and $\Delta_{i+N}\zeta=\Delta_i\zeta$, by~(2), 
\[
    \Delta_\ast P_\xi^0  = \frac1N \sum_{x\in\zeta\cap[0,L)} \delta_{\Delta(\tau_x\zeta)}.
\]
Evaluating this measure on $\left\{u_{j_1}\le t_1,\ldots,u_{j_m}\le t_m\right\}\subset[0,\infty)^\mathbb{Z}$ gives 
\[ 
    \frac1N \#\left\{ 1\le i\le N: \Delta_{i+j_1}\zeta\le t_1,\ldots, \Delta_{i+j_m}\zeta\le t_m \right\},
\] 
as claimed.
\end{proof}

\section{Point process limit of saddle connection directions}

From now on, we fix an $\mathrm{SL}(2,\mathbb R)$-invariant ergodic probability measure $\mu$ on $\mathcal{H}^1(\kappa)$, together with its support
\[
    \mathcal M=\operatorname{supp}\mu.
\]
We will work under the following assumption mentioned in the introduction, and indicate explicitly each use:
\begin{enumerate}[label=($\ast$), ref=($\ast$)]
    \item $\mu$-generic points have circle limit measure $\mu$.
\end{enumerate}

\subsection{Direction point process and slope point process}\label{subsec:direction-and-slope-pp}

We first recall the direction point processes from the introduction.

For $\omega\in\mathcal H(\kappa)$ and $R>0$, $\Theta_\omega^R\subset[0,2\pi)$ denotes the set of distinct directions of saddle connections on $\omega$ of length at most $R$ and $D_R(\omega) = \#\Theta^R_\omega$. Let
\[
    \widetilde\Theta^R_\omega
    :=
    \Theta^R_\omega+2\pi\mathbb Z
\]
be its periodic extension to $\mathbb R$.

The \emph{direction point process} associated with $\omega$ and $R$ is the random variable
\[
    \xi_\omega^R: [0,2\pi)\longrightarrow \mathcal N(\mathbb R),\qquad
    \xi_\omega^R(\theta)
    :=
    \sum_{\varphi\in\widetilde\Theta_\omega^R}
    \delta_{\varphi-\theta},
\]
where $\theta\in[0,2\pi)$ is distributed according to the probability measure $\frac{1}{2\pi}\mathrm{Leb}$. Write  
\[
P_{\omega, R}:= \frac{1}{2\pi} \left(\xi_\omega^R\right)_\ast \mathrm{Leb}
\]
for its distribution on $\mathcal N(\mathbb R)$.

\begin{lemma}\label{lem:angle-point-process}
    For every $\omega\in\mathcal H(\kappa)$ and $R>0$, the direction point process $\xi^R_\omega$ is simple and stationary with intensity $\frac{D_R(\omega)}{2\pi}$.
\end{lemma}

\begin{proof}
Since $\Theta^R_\omega$ consists of distinct directions, $\xi^{R}_{\omega}$ is simple.
    
Set 
\[
\zeta^{R}_{\omega}=\sum_{\varphi\in\widetilde\Theta^R_\omega}\delta_{\varphi}
\]
Then $\zeta^{R}_{\omega}$ is $2\pi$-periodic and 
\[
\xi^{R}_{\omega}(\theta) = \tau_{\theta}\zeta^{R}_{\omega}.
\]
Moreover,
\[
\zeta^{R}_{\omega}([0,2\pi)) = \# \Theta^R_\omega = D_R(\omega).
\]
By Lemma~\ref{lem:periodicPP}~(1), $\xi^{R}_{\omega}$ is stationary with intensity $\frac{D_R(\omega)}{2\pi}$.
\end{proof}

We next extend the slope point processes from the introduction to $\mathcal M$.

Let $\Lambda_\omega^{\mathrm{red}}\subset \Lambda_\omega$ be the subset obtained by taking the shortest holonomy vector in each ray in $\mathbb{R}^2$ containing a vector of $\Lambda_\omega$, where ``$\mathrm{red}$" stands for ``reduced". This choice of holonomy vectors is used in~\cite{UyanikWork2016} to study gap distributions for Veech surfaces. We use the terminology ``reduced" instead of ``short" as in~\cite{UyanikWork2016} so as not to conflict with ``short" as in~\cite{masur2025lengthssaddleconnectionsrandom}. 

For later use, we record the following properties of $\Lambda^{\mathrm{red}}_\omega$. For $f\in B_c(\mathbb{R}^2)$, define the transform $\widehat{f}^{\mathrm{red}}$ by
\[
\widehat{f}^{\mathrm{red}}(\omega):= \sum_{v\in \Lambda^{\mathrm{red}}_\omega}f(v),\qquad \omega\in\mathcal{H}(\kappa).
\]

\begin{lemma}\label{lem:SV-quadratic-short-holonomy-vectors}
   Then there exists a constant $c^{\mathrm{red}}(\mu)>0$ such that
    \[
        \int_{\mathcal M}
        \widehat f^{\mathrm{red}}(\omega)\,d\mu(\omega)
        =
        c^{\mathrm{red}}(\mu)
        \int_{\mathbb R^2} f(v)\,d\operatorname{Leb}_2(v)
    \]
    for every $f\in B_c(\mathbb R^2)$, and
    \[
        \lim_{R\to\infty}
        \frac{1}{\pi R^2}
        \#\{v\in\Lambda_\omega^{\mathrm{red}}:|v|\le R\}
        =
        c^{\mathrm{red}}(\mu)
    \]
    for $\mu$-almost every $\omega\in\mathcal M$.

    Moreover, when $\mu$ is the Masur--Veech measure $\mu_{\mathcal C}$ on a connected component $\mathcal C$ of $\mathcal{H}^1(\kappa)$, $\Lambda_\omega^{\mathrm{red}}=\Lambda_\omega$ for $\mu_{\mathcal C}$-almost every $\omega$, and $c^{\mathrm{red}}(\mathcal C)=c(\mathcal C)$.
\end{lemma}

The asymptotic conclusion equivalently implies that
\[
\lim_{R\to\infty}\frac{1}{\pi R^2}D_R(\omega) = c^{\mathrm{red}}(\mu)
\]
for $\mu$-almost every $\omega\in\mathcal{M}$.

\begin{proof}
    Let $B(0,R)\subset\mathbb R^2$ denote the closed ball of radius $R$
    centered at the origin. The Siegel--Veech formula and the quadratic asymptotics follow from~\cite{EskinMasur2001}. More precisely, to apply their results to $\Lambda_\omega^{\mathrm{red}}$, we verify the following three assumptions listed in~\cite[Section~2]{EskinMasur2001}:
    \begin{enumerate}
        \item[\textup{(A)}] $A\Lambda_\omega^{\mathrm{red}} = \Lambda_{A\omega}^{\mathrm{red}}$ for every $A\in\mathrm{SL}(2,\mathbb R)$;

        \item[\textup{(B)}] for every $\omega\in\mathcal M$, there exists a constant $c>0$
        such that, for every $R>1$,
        \[
            \#\left(\Lambda_\omega^{\mathrm{red}}\cap B(0,R)\right)
            \le cR^2,
        \]
        where $c$ can be chosen uniformly as $\omega$ varies over any compact
        subset of $\mathcal M$;

        \item[\textup{($C_\mu$)}] there exist $R>0$ and $\varepsilon>0$ such that
        \[
            \omega\longmapsto
            \#\left(\Lambda_\omega^{\mathrm{red}}\cap B(0,R)\right)
        \]
        belongs to $L^{1+\varepsilon}(\mathcal M,\mu)$.
    \end{enumerate}

    The first assumption follows directly from the definition of $\Lambda_\omega^{\mathrm{red}}$.

    For~(B) and~($C_\mu$), note that
    \[
        \#\left(\Lambda_\omega^{\mathrm{red}}\cap B(0,R)\right)
        \le
        \#\left(\Lambda_\omega\cap B(0,R)\right).
    \]
    By~\cite[Theorem~5.4]{EskinMasur2001}, the right-hand side is bounded by a multiple of $R^2$, and hence~(B) follows. At the end of the same section, they also show that for some $\varepsilon>0$,
    \[
        \left(\omega\longmapsto \#\left(\Lambda_\omega\cap B(0,R)\right)\right)
        \in L^{1+\varepsilon}(\mathcal M,\mu).
    \]
    Hence~($C_\mu$) follows as well.

    Therefore, \cite[Theorems~2.1 and~2.2]{EskinMasur2001} give the quadratic asymptotics and the Siegel--Veech formula for $\Lambda_\omega^{\mathrm{red}}$, with a constant $c^{\mathrm{red}}(\mu)\ge0$. Moreover, since
    \[
        \mathcal M
        =
        \bigcup_{R\in\mathbb N}
        \{\omega\in\mathcal M:
        \Lambda_\omega^{\mathrm{red}}\cap B(0,R)\neq\varnothing\},
    \]
    there exists $R>0$ for which the set on the right has positive $\mu$-measure. Taking a nonnegative $f\in C_c^\infty(\mathbb R^2)$ that is positive on $B(0,R)$, we have $\int_{\mathcal M} \widehat f^{\mathrm{red}}(\omega)\,d\mu(\omega)>0$ and hence $c^{\mathrm{red}}(\mu)>0$.

    When $\mu=\mu_{\mathcal C}$ is the Masur--Veech measure on a connected component $\mathcal C$ of $\mathcal{H}^1(\kappa)$, \cite[Proposition~7.4]{EskinMasurZorich2003} implies that, for $\mu_{\mathcal C}$-almost every $\omega\in\mathcal C$, there are no two parallel saddle connections of different lengths. Hence $\Lambda_\omega^{\mathrm{red}}=\Lambda_\omega$ for $\mu_{\mathcal C}$-almost every $\omega$. Comparing the corresponding
    Siegel--Veech formulas gives $c^{\mathrm{red}}(\mathcal C)=c(\mathcal C)$.
\end{proof}

Set
\[
    V:=\{(x,y)\in\mathbb R^2:0<x<1\}.
\]
For $\omega\in\mathcal{M}$, let
\[
    \mathbb S_\omega
    :=
    \left\{
        \frac{y}{x}:(x,y)\in\Lambda_\omega^{\mathrm{red}}\cap V
    \right\}.
\]
The \emph{slope point process} $\xi_\mathcal{M}$ associated with $\mu$ (or, equivalently, with $\mathcal M$) is the random variable
\[
    \xi_{\mathcal M}:\mathcal M\to\mathcal N(\mathbb R),
    \qquad
    \xi_{\mathcal M}(\omega)
    :=
    \sum_{s\in \mathbb S_\omega}
    \delta_{s},
\]
where $\omega\in\mathcal M$ is distributed according to $\mu$. Write
\[
    P_{\mathcal M}:=(\xi_{\mathcal M})_\ast\mu,
    \qquad
    E_{\mathcal M}[\cdot]
    :=
    \int_{\mathcal N(\mathbb R)}
    \cdot\,dP_{\mathcal M}(\zeta)
\]
for its distribution and expectation, respectively.

\begin{lemma}\label{lem:slope-point-process}
     The slope point process $\xi_{\mathcal{M}}$ is simple and stationary with intensity $\frac{c^{\mathrm{red}}(\mu)}{2}$.
\end{lemma}

\begin{proof}
    Since $\mathbb S_\omega$ consists of distinct slopes, $\xi_{\mathcal{M}}$ is simple. The stationarity follows from the horocycle flow invariance of $\mu$. For every $s\in\mathbb R$, 
    $
    h_s = 
    \begin{pmatrix}
    1 & 0\\
    -s & 1
    \end{pmatrix}$
    preserves $\mathbb S_\omega$ and sends a vector of slope $k$ to a vector of slope $k - s$. Then $\tau_s\xi_{\mathcal M}(\omega) = \xi_{\mathcal M}(h_{s}\omega)$, where $\tau_s$ is defined in~\eqref{equ:translate-notation}. Since $\mu$ is $h_s$-invariant, $\xi_{\mathcal M}(h_{s}\omega)$ has the same distribution as $\xi_{\mathcal M}(\omega)$. Hence $\xi_{\mathcal M}$ is stationary.

    Consider $ D:=\{(x,y)\in\mathbb R^2:0<x<1,\ 0\le y\le x\}$. By Lemma~\ref{lem:SV-quadratic-short-holonomy-vectors},
    \[
        E_{\mathcal M}\left[\zeta([0,1])\right]
        =\int_{\mathcal{N}(\mathbb R)}\zeta([0,1])dP_\mathcal{M}(\zeta)
        =\int_{\mathcal M}\#(\Lambda_\omega^{\mathrm{red}}\cap D)\,d\mu(\omega)
        =\frac{c^{\mathrm{red}}(\mathcal C)}2.
    \]
    Thus the intensity of $\xi_{\mathcal M}$ is $\frac{c^{\mathrm{red}}(\mu)}{2}$.
\end{proof}

For a bounded Borel set $B\subset\mathbb R$, we use the notation
\[
    \xi^R_\omega(B):
    \theta\longmapsto\xi^R_\omega(\theta)(B),
    \qquad
    \xi_{\mathcal M}(B):
    \omega\longmapsto\xi_{\mathcal M}(\omega)(B)
\]
for the corresponding random variables.

\subsection{Point process limit in a fixed stratum}\label{subsec:fixed-stratum-convergence}

In this subsection, we prove Theorem~\ref{thm:direction-point-process}. In fact, Theorem ~\ref{thm:direction-point-process} holds whenever $\mu$ satisfies~\ref{ass:circle-limit} and so we deduce Theorem ~\ref{thm:direction-point-process} from the following theorem applied to $\mu=\mu_\mathcal C$.

\begin{theorem}\label{thm:fixed-stratum-point-process}
Assume that $\mu$ satisfies~\ref{ass:circle-limit}. For $\mu$-almost every $\omega\in\mathcal M$,
\[
    R^2\xi^{R}_{\omega} 
    \xrightarrow{d} \xi_{\mathcal M},
    \qquad\text{as }R\to\infty.
\]
\end{theorem}

We prove this by computing the finite dimensional distributions of each; see Theorem~\ref{thm:characterization-point-process-convergence}.  Towards this we define the following.

For $R>0$, $\theta\in[0,2\pi)$, and a bounded interval
$I\subset\mathbb R$, define
\[
    N(\omega,\theta,R,I)
    :=
    \#\left\{
    v\in\Lambda_\omega:
    |v|<R,\ 
    \arg v-\theta\in\frac{1}{R^2}I
    \right\}.
\]
For $I=[a,b]\subset\mathbb R$, set
\begin{equation}\label{equ:triangle-in-plane}
    T_I
    :=
    \left\{
    (x,y)\in\mathbb R^2:
    0< x<1,\ 
    ax\leq y\leq bx
    \right\}.
\end{equation}

Then computing the finite dimensional distributions is handled by the following proposition which generalizes Theorem 3.2 of \cite{AthreyaChaika2012}.

\begin{proposition}\label{prop:joint-count-convergence}
Let $I_1,\ldots,I_m$ be pairwise disjoint bounded closed intervals in $\mathbb R$, and let $n_1,\ldots,n_m\in\mathbb Z_{\geq 0}$. Suppose that $\omega\in\mathcal H^1(\kappa)$ has circle limit measure $\mu$ with support $\mathcal M$. Then
\[
\begin{aligned}
    &\lim_{R\to\infty}
    \frac{1}{2\pi}
    \operatorname{Leb}
    \left\{
    \theta\in[0,2\pi):
    N(\omega,\theta,R,I_i)=n_i
    \text{ for }1\leq i\leq m
    \right\}
    \\
    &\qquad =
    \mu
    \left\{
    \eta\in\mathcal M:
    \#\left(\Lambda_\eta\cap T_{I_i}\right)=n_i
    \text{ for }1\leq i\leq m
    \right\}.
\end{aligned}
\]
\end{proposition}

\begin{proof}
    Let $\lambda$ denote the normalized Lebesgue measure on the circle.

 For an interval $I$, consider $$A^\theta _R(I):=\{v\in\mathbb R^2: \|v\|\le R,\arg(v)-\theta \in \frac{1}{R^2}I\}.$$
    The main idea, as in Theorem 3.2 of \cite{AthreyaChaika2012},  is to show that for any interval $I$, $g_{-t}r_{-\theta}A^\theta _R(I)$ is close to $T_I $
    for $R\gg0$ and $t=2\log(R)$. In particular, counting vectors in
    $\Lambda_\omega\cap A^\theta _R(I) $ is close to counting vectors in  $g_{-t}r_{-\theta}\Lambda_\omega \cap T_I.$

    Thus, if we take disjoint intervals $I_1$,...,$I_m$, and for each $i$, take $f_{n_i}:\mathcal M\to[0,1]$ to be the indicator function of 
    $$\mathcal M_{I_i, n_i} := \{\omega \in \mathcal{M} : \vert{}\Lambda_\omega \cap T_{I_i}\vert{} = n_i\},$$
    then we can relate our counting question to a dynamical one via
    \begin{equation}\label{eqn: key identity}
        \lambda(\theta : N(\omega, \theta, R, I_i)= n_i
    \text{ for }1\leq i\leq m) \approx \int_0^{2\pi} \prod_{i=1}^mf_{n_i}(g_{-t} r_{-\theta}\omega) d\lambda(\theta) 
    \end{equation}
      where $\approx$ denotes that the difference goes to 0 as $R\to \infty$ (or equivalently $t\to\infty$). By using our assumption on expanding circless~\ref{ass:circle-limit}, we would have that 
     \begin{align*}
     \lim_{R\to\infty} \lambda(\theta : N(\omega, \theta, R, I_i)= n_i
    \text{ for }1\leq i\leq m)& =\lim_{t\to\infty} \int_0^{2\pi} \prod_{i=1}^mf_{n_i}(g_{-t} r_{-\theta}\omega) d\lambda(\theta)\\
    &= \int_\mathcal M \prod_{i=1}^mf_{n_i}(\omega) d\mu\\
    &=\mu
    \left\{
    \eta\in\mathcal M:
    \#\left(\Lambda_\eta\cap T_{I_i}\right)=n_i
    \text{ for }1\leq i\leq m
    \right\}
    \end{align*}
    as required.

   The majority of the proof amounts to showing the $\approx$ above for a single interval, that is, Equation~\eqref{eqn: key identity} for a single interval. To this end, define a measure by $\int_{\mathcal{M}}f d\nu_{t} := \int_0^{2\pi}f(g_{-t} r_{-\theta}\omega) d\lambda(\theta)$. We suppress the dependence on $\omega.$

     Let $I=[a,b]$ where $b\ge a\ge 0.$ The case where we allow negative slopes follows by symmetry. Let $E_t$ denote the symmetric difference of $g_{-t}r_{-\theta}A^\theta_R(I)$ and $T_I$ and note that this does not depend on $\theta$ because we have rotated the wedge $r_{-\theta}A^\theta_R(I)$ by $-\theta$.  We will take $t=2\log(R)$ throughout.
For $t_0\in\mathbb R$, consider 
$$E_{t_0}^*:=\cup_{t\ge t_0}E_t.$$ We claim that as $t_0\to\infty$, the Lebesgue measure of this set, that we denote by $m$, tends to 0.

Indeed, note that $E_{t_0}^*\subset U_{t_0}\cup H_{t_0}\cup L_{t_0}$ where we choose $t_0$ so large so that $\frac{b}{R_{t_0}^2}<1$ and where
\begin{itemize}
    \item $U_{t_0}$ is the triangle with vertices $(0,0)$,$(1,b)$, and $(1,R_{t_0}^2\tan\left(\frac{b}{R_{t_0}^2}\right))$ that, informally, is the triangle that captures the error at the top between the largest possible slope of the wedges in $E_{t_0}^*$ at the top of $E_{t_0}$ and the top of $T_I$,
    \item $L_{t_0}$ is the triangle with vertices $(0,0)$,$(1,a)$, and $(1,R_{t_0}^2\tan\left(\frac{a}{R_{t_0}^2}\right))$ that, informally, is the triangle that captures the error at the bottom between the smallest possible slope of the wedges in $E_{t_0}^*$ at the bottom of $E_{t_0}$ and the bottom of $T_I$, and
    \item $H_{t_0}$ which is the rectangle  $[\cos(\frac{b}{R_{t_0}^2}),1]\times [0,R_{t_0}^2\tan\left(\frac{b}{R_{t_0}^2}\right)]$ that, informally, is the horizontal rectangle that captures the error of the wedges in $E_{t_0}^*$ between the left most wedge and the boundary at $x=1$.
\end{itemize}

Since these are simple geometric shapes, we may estimate their contribution. We have
\begin{itemize}
    \item $U_{t_0}$ has height of length $R_{t_0}^2\tan\left(\frac{b}{R_{t_0}^2}\right)-b=R_{t_0}^2\left(\tan\left(\frac{b}{R_{t_0}^2}\right)-\frac{b}{R_{t_0}^2}\right)$ and base of length 1 so that   
    $$m(U_{t_0})=\frac{R_{t_0}^2}{2}\left (\tan\left(\frac{b}{R_{t_0}^2}\right)-\frac{b}{R_{t_0}^2}\right)\le  \frac{R_{t_0}^2}{2}\left(\frac{b}{R_{t_0}^2}\right)^3 =\frac{b^3}{2R_{t_0} ^4},$$
    \item a similar argument on $L_{t_0}$ shows that  $$m(L_{t_0})\le\frac{a^3}{2R_{t_0} ^4},$$ and
    \item $H_{t_0}$ has a base of length of $1-\cos(\frac{b}{R_{t_0}^2})\le \frac{1}{2}\left(\frac{b}{R_{t_0}^2}\right)^2 $ and height of length $R_{t_0}^2\tan\left(\frac{b}{R_{t_0}^2}\right)\le2R_{t_0}^2\frac{b}{R_{t_0}^2}$ to see
    $$m(H_{t_0})\le \frac{b^3}{R_{t_0}^4}.$$
\end{itemize}
 As such, since $E_{t_0}^*\subset U_{t_0}\cup H_{t_0}\cup L_{t_0}$, then the volume of $E_{t_0}^*$ tends to 0 as $t_0\to\infty.$

Now we prove
$$\lim_{t\to\infty }\nu_t\{\omega:\hat{F_t}(\omega)=k\}=\lim_{t\to\infty }\nu_t\{\omega:\hat{\chi}_{T_I}(\omega)=k\}$$
where $F_t$ is the indicator function of $g_{-t}r_{-\theta}A^\theta_R(I)$. We note that, by construction, $\hat{F_t}(g_{-t}r_{-\theta}\omega)=N(\omega,\theta,R,I)$.

Let $I_t=\{\omega\in\mathcal M: \hat F_t\ne\hat \chi_{T_I}\}$ and $I^*_{t_0}=\cup _{t\ge t_0}I_t$.

Note if $\omega\in I_{t_0}^*$, then there is some $t\ge t_0$ such that $\hat F_t(\omega)\ne\hat \chi_{T_I}(\omega)$ and so there is some holonomy vector so that $v\in E_t$. Hence, by the Siegel-Veech formula,

$$\mu(I_{t_0}^*)\le \int_\mathcal M \hat{\chi}_{E_{t_0}^*}\,d\mu= c(\mathcal M)m(E_{t_0}^*).$$

Fix $\varepsilon>0$ and, by potentially choosing $t_0$ larger, assume $m(E_{t_0}^*)<\varepsilon/c(\mathcal M)$ for all $\theta\in[0,2\pi)$, and thus,
$$\mu(I_{t_0}^*)<\varepsilon.$$
Since $\nu_t\to\mu$, we can choose $t_1>t_0$ so that for all $t>t_1,$ we have $\nu_t(I_{t}^*)\le 2\varepsilon$. Now, since $I_t\subseteq I_{t_0}^*$ for $t>t_1>t_0$, then $$\nu_t(I_t)\le \nu_t(I_{t_0}^*)\le 2\varepsilon.$$
We have
$$\nu_t(\omega: \hat F_t(\omega) = k) = \nu_t(\omega \notin I_t : \hat F_t = k) + \nu_t(\omega \in I_t : \hat F_t(\omega) = k)$$ 
and that the second term is bounded above by $\nu_t(I_t)$ and, hence, by $2\varepsilon.$ The first term can be written as 
\begin{align*}
\nu_t(\omega \notin I_t : \hat F_t = k)&=\nu_t(\omega \notin I_t : \hat{\chi}_{T_I}(\omega) = k)\\
&=\nu_t(\omega: 
\hat{\chi}_{T_I}(\omega) = k) - \nu_t(\omega \in I_t : \hat{\chi}_{T_I}(\omega) = k).
\end{align*}

Recalling our goal of 
$$\lim_{t\to\infty }\nu_t\{\omega:\hat{F_t}(\omega)=k\}=\lim_{t\to\infty }\nu_t\{\omega:\hat{\chi_{T_I}}(\omega)=k\},$$
we may consider their difference and use the above observations to deduce
$$|\nu_t\{\omega:\hat{F_t}(\omega)=k\}-\nu_t\{\omega:\hat{\chi_{T_I}}(\omega)=k\}| \le \nu_t\{\omega\in I_t:\hat{\chi}_{T_I}(\omega)=k\}+2\varepsilon.$$
By using the bound on $\nu_t(I_t)$ we get
$$|\nu_t(\omega: \hat F_t(\omega) = k) - \nu_t(\omega: \hat{\chi}_{T_I}(\omega) = k)|\le 4\varepsilon.$$ Passing to the limit, we obtain that the limits can differ by no more than $4\varepsilon$ and hence, must be the same since $\varepsilon$ is arbitrary.

Now assume that $I_1,\ldots,I_m$ are pairwise disjoint bounded closed intervals in $\mathbb R$. The change from one interval to multiple is minor. Let $F_{t,i}$ denote the indicator function of $g_{-t}r_{-\theta}A_R ^\theta(I_i)$ and $I_{t,i}=\{\omega\in \mathcal M: \hat{F}_{t,i}\ne \hat{\chi}_{T_{I_i}}\}$. Let $\varepsilon>0$. By the single interval case, we have that there is some $t_1(i)$ such that  for $t>t_1(i)$, we have $\nu_t(I_{t,i})\le2\varepsilon.$ Letting $t_1$ be the max of the finitely many $t_1(i)$, we have that, for $\tilde I_t :=\cup_{i=1}^m I_{t,i}$, the bound 
$$\nu_t(\tilde I_t)\le \sum_{i=1} ^m \nu_t(I_{t,i})\le2m\varepsilon.$$

Then using this bound as in the single interval case, we can show that the difference

$$|\nu_t(\omega:\hat{F}_{t,i}(\omega)=n_i,\forall i)-\nu_t(\omega:\hat{\chi}_{T_{I_i}}(\omega)=n_i,\forall i)|\le 4m\varepsilon.$$
As such, by taking limits, we have the result
$$\lim_{t\to\infty}\nu_t(\omega:\hat{F}_{t,i}(\omega)=n_i,\forall i) =\lim_{t\to\infty}\nu_t(\omega:\hat{\chi}_{T_{I_i}}(\omega)=n_i,\forall i) .$$

Of course, $\hat{F}_{t,i}(g_{-t}r_{-\theta}\omega)=N(\omega,\theta,R,I_i)$ and so
$$        \lambda(\theta : N(\omega, \theta, R, I_i)= n_i
    \text{ for }1\leq i\leq m) =\nu_t(\omega:\hat{F}_{t,i}(\omega)=n_i,\forall i)$$
and 
$$\nu_t(\omega:\hat{\chi}_{T_{I_i}}(\omega)=n_i,\forall i)=\int_\mathcal M\prod_{i=1}^mf_{n_i}\,d\nu_t=\int_0^{2\pi} \prod_{i=1}^mf_{n_i}(g_{-t} r_{-\theta}\omega) d\lambda(\theta).$$ Thus,  this establishes Equation~\eqref{eqn: key identity} and finishes the proof.
\end{proof}

\begin{proof}[Proof of Theorem~\ref{thm:fixed-stratum-point-process}]

This follows by noticing that Proposition~\ref{prop:joint-count-convergence} amounts to computing the finite dimensional distributions of the point processes $R^2\xi^R_\omega$  and $\xi_\mathcal M.$ We supply some details for completeness. 

Recall that  we define direction point process by
\[
\xi_\omega^R(\theta):= 
\widetilde{\Theta}^R_\omega - \theta 
= 
\left\{\varphi-\theta: \varphi\in\widetilde\Theta_\omega^R\right\}.
\]
where $\omega\in\mathcal M$, $R>0$, and $\widetilde \Theta_\omega^R$ is the $2\pi$-periodic extension of saddle connection directions on $\omega$ of length at most $R$. 

It suffices to consider intervals since they generate the Borel $\sigma$-algebra.  If $I_1,\ldots,I_m$ are pairwise bounded closed intervals, then the finite dimensional distribution of $R^2\xi^R_\omega$ is given by
\begin{align*}\frac{1}{2\pi}\text{Leb}(\{\theta : (R^2\xi^R_\omega (\theta) )(I_i)&=n_i, \text{ for }1\le i\le m\})= \frac{1}{2\pi}\text{Leb}(\{\theta : \xi^R_\omega (\theta )\left(\frac{1}{R^2}I_i\right)=n_i, \text{ for }1\le i\le m\})\\
&=\frac{1}{2\pi}\text{Leb}(\{\theta : \#\{\varphi-\theta\in\frac{1}{R^2}I_i:\varphi \in \widetilde\Theta_\omega^R \}= n_i, \text{ for }1\le i\le m\})\\
&=\frac{1}{2\pi}\text{Leb}(\{\theta : \#\{\varphi\in\frac{1}{R^2}I_i+\theta:\varphi \in \widetilde\Theta_\omega^R \}= n_i, \text{ for }1\le i\le m\})\\
&=\frac{1}{2\pi}\text{Leb}(\{\theta : \#\{v\in\Lambda_\omega:|v|<R, \varphi=\arg v\in \frac{1}{R^2}I_i+\theta \}= n_i, \text{ for }1\le i\le m\})\\
&=\frac{1}{2\pi}\text{Leb}(\{\theta:N(\omega,\theta,R,I_i)=n_i \text{ for }1\le i\le m\})
\end{align*}
where the second to last equality follows since $\varphi\in \widetilde\Theta_\omega^R$ means exactly that there is a saddle connection $v$ of length at most $R$ about $\theta$.

Similarly, we recall that, for $V=(0,1)\times\mathbb R$ and  $\omega$ distributed according to $\mu_{\mathcal M}$, we defined the slope point process $\xi_{\mathcal M}$ by
\[
\xi_{\mathcal M}(\omega) = \left\{ \frac{v_y}{v_x}: v=(v_x,v_y)\in\Lambda^{\mathrm{red}}_\omega\cap V \right\}.
\]

For pairwise disjoint closed and bounded intervals $I_1,\ldots,I_m$ in $\mathbb R$, if
 the horizontal of the saddle connection $v$ is between 0 and 1 and that the slope of $v$ is in one of the $I_i$, we have that $v\in T_{I_i}$. As such, the finite dimensional distribution of $\xi_\mathcal M$ is given by
 $$\mu\{\eta\in\mathcal M:(\xi_\mathcal M)(\eta)(I_i)=n_i\text{ for }1\le i\le m \}=\mu
    \left\{
    \eta\in\mathcal M:
    \#\left(\Lambda_\eta\cap T_{I_i}\right)=n_i
    \text{ for }1\leq i\leq m
    \right\}.$$

We can thus apply Proposition~\ref{prop:joint-count-convergence} and since the finite dimensional distributions of $R^2\xi_\omega ^R$ converge to $\xi_\mathcal M$, Theorem~\ref{thm:characterization-point-process-convergence} gives that     
$$R^2\xi^{R}_{\omega} 
    \xrightarrow{d} \xi_{\mathcal M},
    \qquad\text{as }R\to\infty$$
 
\end{proof}

\section{Gap distributions}

\subsection{Direction and slope gap distributions}
Recall that $\mu$ is an ergodic measure on $\mathcal{H}^1(\kappa)$ and $\mathcal{M}$ is its support. For $\omega\in\mathcal M$ and $R>0$, $\xi^R_\omega$ denotes the direction point process and $P_{\omega,R}$ denotes its distribution on $\mathcal{N}_s(\mathbb R)$. 

Write 
\[ 
\widetilde\Theta_\omega^R = \{\theta_i:i\in\mathbb Z\}, \qquad \theta_i<\theta_{i+1}, \qquad \theta_0\le0<\theta_1,
\]
and extend the gaps in~\eqref{equ:original-gap-sequence} periodically by setting 
\[
\Delta_i^\omega:=\theta_{i+1}-\theta_i, \qquad i\in\mathbb Z.
\]

We first relate these gaps to the gap-sequence distribution of $\xi^R_\omega$; see Section~\ref{subsec:Palm-distribution}. Let $P_{\omega,R}^0$ denote the Palm distribution of $\xi^R_\omega$ on $\mathcal{N}^0_s(\mathbb R)$.

\begin{lemma}\label{lem:expression-related-to-counting}
    For every $\omega\in\mathcal{M}$, every $R>0$ such that $D_R(\omega)>0$, every $m\in\mathbb N$, $j_1,\ldots,j_m\in\mathbb Z$ and $t_1,\ldots,t_m\ge 0$,
    \begin{equation}\label{equ:expression-related-to-counting}
        \begin{aligned}
            &P_{\omega,R}^0\left(\Delta_{j_1}\xi\le t_1,\ldots, \Delta_{j_m}\xi\le t_m\right)\\
            &=
            \frac{1}{D_R(\omega)} \#\left\{ 1\le i\le D_R(\omega): \Delta^\omega_{i+j_\ell}\le t_\ell \text{ for }1\le \ell\le m \right\}.
        \end{aligned}
    \end{equation}
\end{lemma}

\begin{proof}
    In the proof of Lemma~\ref{lem:angle-point-process}, we showed that $\xi^R_\omega=\tau_\theta\zeta^R_\omega$, where $\theta$ is uniformly distributed on $[0,2\pi)$. The claim now follows directly from Lemma~\ref{lem:periodicPP}~(3).
\end{proof}

We now turn to the gap-sequence distribution of the slope point process $\xi_\mathcal{M}$. Recall that $P_\mathcal{M}$ and $E_{\mathcal M}$ denote its distribution and expectation on $\mathcal{N}_s(\mathbb R)$. We write
\[
P_{\mathcal M}^0\quad \text{and}\quad E_{\mathcal M}^0
\]
for the Palm distribution of $\xi_\mathcal{M}$ and its expectation on $\mathcal{N}_s^0(\mathbb R)$, respectively.

The gap-sequence distribution of $\xi_\mathcal{M}$ is $\Delta_\ast P_{\mathcal M}^0$ by definition. For $m\in\mathbb N$ and $j_1,\ldots,j_m\in\mathbb Z$, following~\eqref{equ:finite-dimensional-distribution}, define
\begin{equation}\label{equ:slope-pp-distribution-function}
    G_{\mathcal{M};\, j_1,\ldots,j_m}(t_1,\ldots,t_m):= P^0_{\mathcal{M}}(\Delta_{j_\ell}\zeta\le t_\ell,\ 1\le \ell\le m),
\end{equation}
where $t_1,\ldots,t_m\in\mathbb R$.

For later use, we record an expression for this distribution function. Recall that
\[
    \mathbb S_\omega
    :=
    \left\{
        \frac{y}{x}:(x,y)\in\Lambda_\omega^{\mathrm{red}}\cap V
    \right\}.
\]
We write this as
\[
\mathbb S_\omega=(s_i(\omega))_{i\in\mathbb Z},
\qquad
s_i(\omega)<s_{i+1}(\omega),
\qquad
s_0(\omega)\le 0<s_1(\omega)
\]
and have that the slope point process may be written as $\xi_{\mathcal M}(\omega) = \sum_{i\in\mathbb Z}\delta_{s_i(\omega)}$.

Let $\lambda_\mathcal{M}$ denote the intensity of $\xi_\mathcal{M}$. Taking an interval $I$ with $\operatorname{Leb}(I)=1$ and setting $B:=I$ in the refined Campbell formula~\eqref{equ:computation-Palm-distribution}, we obtain
\begin{equation}\label{equ:rephrase-slope-gap-distribution} 
\begin{aligned} 
&G_{\mathcal M;\,j_1,\ldots,j_m}(t_1,\ldots,t_m) \\ 
&\quad
= 
\frac{1}{\lambda_{\mathcal M}} 
E_{\mathcal M} \left[ \sum_{s_i(\omega)\in I} 
\boldsymbol{1}_{
    \left\{
    \Delta_{j_\ell}\zeta\le t_\ell
    \text{ for }1\le\ell\le m
    \right\}
}
\left(\tau_{s_i(\omega)}\xi_{\mathcal M}(\omega)\right) \right] \\ 
&\quad
= 
\frac{1}{\lambda_{\mathcal M}} 
\int_{\mathcal M} 
\#\big\{s_i(\omega)\in I: 
s_{i+j_\ell+1}(\omega) - s_{i+j_\ell}(\omega) 
\le t_\ell \text{ for }1\le\ell\le m \big\}
\,d\mu(\omega).
\end{aligned}
\end{equation}

\subsection{Existence of joint gap distributions}\label{subsec:convergence-joint-gap-distributions}

We first record the following convergence of intensities.

\begin{lemma}\label{lem:convergence-of-intensities}
For $\mu$-almost every $\omega\in\mathcal{M}$, the intensity of $R^2\xi^{R}_\omega$ converges to the intensity of $\xi_{\mathcal{M}}$ as $R\to\infty$.
\end{lemma}

\begin{proof}
By Lemma~\ref{lem:angle-point-process} and the properties of rescaled point processes recalled at the end of Section~\ref{subsec:Palm-distribution}, the intensity of $R^2\xi^{R}_\omega$ is $\frac{D_R(\omega)}{2\pi R^2}$. Since $D_R(\omega)$ counts the number of distinct directions of saddle connections of length at most $R$,
\[
D_R(\omega) = \#\left(\Lambda_\omega^{\mathrm{red}}\cap B(0,R)\right).
\]
Hence, Lemma~\ref{lem:SV-quadratic-short-holonomy-vectors} implies that for $\mu$-almost every $\omega$,
\[
    \lim_{R\to\infty}\frac{D_R(\omega)}{2\pi R^2}
    =
    \frac{c^{\mathrm{red}}(\mu)}{2}.
\]
By Lemma~\ref{lem:slope-point-process}, the right-hand side is precisely the intensity of $\xi_{\mathcal M}$.
\end{proof}

We now prove the convergence in Theorem~\ref{thm:joint-gap-distributions}(1) at every continuity point of $G_{\mathcal M;\,j_1,\ldots,j_m}$. We will see in the next section that $G_{\mathcal M;\,j_1,\ldots,j_m}$ is always continuous. 

\begin{proposition}\label{prop:convergence-joint-gap-distributions}
    Assume that $\mu$ satisfies~\ref{ass:circle-limit}. For $\mu$-almost every $\omega\in\mathcal{M}$, every collection of distinct $j_1,\ldots,j_m\in\mathbb Z$, if $G_{\mathcal M;\,j_1,\ldots,j_m}$ is continuous at $(t_1,\ldots,t_m)\in[0,\infty)^m$, then
    \begin{equation}\label{equ:convergence-joint-gap-distributions}
        \begin{aligned} 
        &\lim_{R\to\infty} \frac{1}{D_R(\omega)} 
        \#\left\{ 1\le i\le D_R(\omega): R^2\Delta_{i+j_\ell}\le t_\ell \text{ for }1\le \ell\le m \right\} \\ 
        &\hspace{5cm} 
        =
        G_{\mathcal{M};\, j_1,\ldots,j_m}(t_1,\ldots,t_m). 
        \end{aligned}
    \end{equation}
\end{proposition}

\begin{proof}
By Theorem~\ref{thm:fixed-stratum-point-process}, which uses Assumption~\ref{ass:circle-limit}, for $\mu$-almost every $\omega\in\mathcal M$,
\[
    R^2\xi^R_\omega
    \xrightarrow{d}
    \xi_{\mathcal M}
    \qquad\text{as }R\to\infty.
\]
By Lemma~\ref{lem:convergence-of-intensities}, their intensities converge as well. Recall that $P_{R^2\xi^R_\omega}^0$ denotes the Palm distribution of
$R^2\xi^R_\omega$. By Lemma~\ref{lem:continuous-mapping-theorem}, we have the weak convergence:
\begin{equation}\label{equ:large-R-gap-convergence}
    \Delta_\ast P_{R^2\xi^R_\omega}^0
    \longrightarrow
    \Delta_\ast P_{\mathcal M}^0.
\end{equation}

Since $G_{\mathcal{M};\, j_1,\ldots,j_m}$ is continuous  at $(t_1,\ldots,t_m)\in[0,\infty)^m$, the Portmanteau theorem implies that as $R\to\infty$,
\[
    \begin{aligned}
    \left(\Delta_\ast P_{R^2\xi^R_\omega}^0\right)
    \left(
    u_{j_1}\le t_1,\ldots,u_{j_m}\le t_m
    \right)
    &= 
    \left(\Delta_\ast P_{\omega,R}^0\right)
    \left(
    R^{2}u_{j_1}\le t_1,\ldots, R^{2}u_{j_m}\le t_m
    \right)\\
    &\longrightarrow
    G_{\mathcal{M};\, j_1,\ldots,j_m}(t_1,\ldots,t_m).
\end{aligned}
\]
By Lemma~\ref{lem:expression-related-to-counting}
\[
\begin{aligned}
&\left(\Delta_\ast P_{\omega,R}^0\right)
\left(
R^{2}u_{j_1}\le t_1,\ldots, R^{2}u_{j_m}\le t_m
\right)\\
&=
\frac{1}{D_R(\omega)} 
\#\left\{ 1\le i\le D_R(\omega): R^2\Delta_{i+j_\ell}\le t_\ell \text{ for }1\le \ell\le m \right\}.
\end{aligned}
\]
Then the above convergence gives ~\eqref{equ:convergence-joint-gap-distributions}.
\end{proof}

\section{Regularity of gap distributions}

\subsection{Lipschitz continuity of gap distributions}\label{subsec:gap-regularity}

We continue to work with the ergodic measure $\mu$ on $\mathcal{H}^1(\kappa)$ and its support $\mathcal{M}$.

We study the regularity of the finite-dimensional distribution functions $G_{\mathcal{M};\, j_1,\ldots,j_m}$, as promised in the proof of Proposition~\ref{prop:convergence-joint-gap-distributions}. We begin with the one-dimensional distribution $G_{\mathcal M;\,0}$.

\begin{proposition}\label{prop:one-dimensional-Lipschitz}
    The function $G_{\mathcal M;\,0}$ is continuous on $[0,\infty)$, and for every $0<a\le t<t'$,
    \[
    0\le G_{\mathcal M;\,0}(t')-G_{\mathcal M;\,0}(t)
    \le
    2e
    \left(2+\frac{1}{a}\right)
    (t'-t).
    \]
\end{proposition}

In what follows, we use the abbreviation $\lambda = \lambda_\mathcal{M}$ for the intensity of $\xi_{\mathcal M}$. Recall that the slope point process $\xi_{\mathcal M}$ is defined by
\[
    \xi_{\mathcal M}(\omega)
    =
    \sum_{s\in\mathbb S_\omega}\delta_s,
\]
for $\omega\in\mathcal M$, where $\mathbb S_\omega$ is the set of slopes of holonomy vectors in $\Lambda^{\mathrm{red}}_\omega\cap V$, with $V=\{(x,y)\in\mathbb R^2:0<x<1\}$.

Recall that $\mathbb{S}_{\omega}$ denotes the set of slopes of holonomy vectors in $\Lambda^{\mathrm{red}}_\omega\cap V$, and $\xi_\mathcal M(\omega)$ is the corresponding counting measure. We write $\mathbb{S}_{\omega}=(s_i(\omega))_{i\in\mathbb Z}$ with $s_i(\omega)<s_{i+1}(\omega)$ and $s_0(\omega)\le 0<s_1(\omega)$.

Fix $t>0$ and $h\ge 0$. Taking $I:=(0,1)$ in~\eqref{equ:rephrase-slope-gap-distribution}, we have
\[
G_{\mathcal M;0}(t + h) - G_{\mathcal M;0}(t) = 
\frac{1}{\lambda} 
E_\mathcal{M}
\big[
N_{t,h}(\omega)
\big],
\]
where
\[
N_{t,h}(\omega)
:= 
\#A_{t,h}(\omega),
\qquad
A_{t,h}(\omega)
:=
\big\{ 
s_i(\omega)\in (0,1): s_{i+1}(\omega)-s_{i}(\omega)\in (t,t+h]
\big\}.
\]
Since $\mu$ is invariant under 
$g_r
=
\begin{pmatrix}
e^{r/2} & 0\\
0 & e^{-r/2}
\end{pmatrix}
$, we may average over $r\in[0,1]$ and obtain
\begin{equation}\label{equ:1-step-rewrite-difference}
G_{\mathcal M;0}(t + h) - G_{\mathcal M;0}(t) = 
\frac{1}{\lambda} 
E_\mathcal{M}
\left[
\int_0^1
N_{t,h}\left(g_{r}\,\omega\right)\,
dr
\right].
\end{equation}

We now estimate the integral in~\eqref{equ:1-step-rewrite-difference}. Write a vector $v\in\mathbb R^2$ as $v=(v_x,v_y)$, and if $v_x\ne 0$, denote its slope by
\[
s(v):=\frac{v_y}{v_x}.
\]
Then $s(g_rv) = e^{-r}s(v)$.

We say that two slopes $s<s'$ are \emph{consecutive} in $\mathbb{S}_{\omega}$ if $s=s_i(\omega)$ and $s' = s_{i+1}(\omega)$ for some $i\in\mathbb Z$.

For $v,w\in\Lambda_\omega^{\mathrm{red}}$, let
\[
R_{t,h}^\omega(v,w)
:=
\left\{
r\in[0,1]:
\begin{array}{l}
s(g_rv)\in A_{t,h}(g_r\omega) \text{ and } s(g_rv)<s(g_rw)\\
\text{ are consecutive slopes in } \mathbb S_{g_r\omega}
\end{array}
\right\}.
\]
Define
\[
\mathcal E_{t,h}(\omega) := \left\{ (v,w)\in\left(\Lambda_\omega^{\mathrm{red}}\right)^2: R_{t,h}^\omega(v,w)\ne\varnothing \right\}.
\]
For $(v,w)\in \mathcal E_{t,h}(\omega)$ and $r\in R_{t,h}^\omega(v,w)$, the condition $s(g_rv)\in A_{t,h}(g_r\omega)$ implies that 
\begin{equation}\label{equ:consecutive-gap-condition}
0<e^{-r}s(v)<1, \qquad t<e^{-r}\bigl(s(w)-s(v)\bigr)\le t+h. 
\end{equation}

\begin{lemma}\label{lem:bound-integral-of-counting}
    For $t>0$ and $h\ge 0$, we have
    \[
    \int_0^1
    N_{t,h}\left(g_{r}\,\omega\right)\,
    dr
    \le
    \#\mathcal{E}_{t,h}(\omega)\frac{h}{t}.
    \]
\end{lemma}

\begin{proof}
    For each $(v,w)\in \mathcal{E}_{t,h}(\omega)$, we have
    \[
    R_{t,h}^\omega(v,w)\subset I_{t,h}(v,w):= \{r:\,t<e^{-r}(s(w) - s(v))\le t+h\}.
    \]
    It follows that
    \[
    N_{t,h}(g_r\omega) = \# A_{t,h}(g_r\omega) = \sum_{(v,w)\in\mathcal{E}_{t,h}(\omega)}\chi_{R^\omega_{t,h}(v,w)}(r)\le \sum_{(v,w)\in \mathcal{E}_{t,h}(\omega)}\chi_{I_{t,h}(v,w)}(r),
    \]
    and hence
    \[
    \int_0^1
    N_{t,h}\left(g_{r}\,\omega\right)\,
    dr
    \le
    \sum_{(v,w)\in \mathcal{E}_{t,h}(\omega)}
    \mathrm{Leb}(I_{t,h}(v,w)).
    \]
    Note that
    \[
    \mathrm{Leb}(I_{t,h}(v,w))
    =
    \log \left(\frac{t+h}{t}\right).
    \]
    Therefore,
    \[
    \int_0^1
    N_{t,h}\left(g_{r}\,\omega\right)\,
    dr
    \le 
    \sum_{(v,w)\in \mathcal{E}_{t,h}(\omega)}
    \log \left(\frac{t+h}{t}\right) 
    \le
    \#\mathcal{E}_{t,h}(\omega)\frac{h}{t}. 
    \]
\end{proof}

It remains to bound $\#\mathcal{E}_{t,h}(\omega)$. For that, we introduce a domain
\[
K_{t,h}
:=
\left\{
(x,y)\in\mathbb R^2:
0<x<1,\; 0\le\frac{y}{x}\le e(1+t + h)
\right\}.
\]

\begin{lemma}\label{lem:bound-pair-subset}
    For $t\ge 0$ and $h\ge 0$, we have
    \[
    \#\mathcal{E}_{t,h}(\omega)
    \le 
    2\,\#
    \left(
    \Lambda^{\mathrm{red}}_\omega \cap K_{t,h}
    \right).
    \]
\end{lemma}

\begin{proof}
    We first claim that $\mathcal{E}_{t,h}(\omega)\subset \left(\Lambda^{\mathrm{red}}_\omega \cap K_{t,h}\right)^2$. Indeed, for $(v,w)\in \mathcal{E}_{t,h}(\omega)$ and $r\in R_{t,h}^\omega(v,w)$, the condition~\eqref{equ:consecutive-gap-condition} implies that
    \[
    0<s(v)<e^r\le e
    \]
    and
    \[
    s(w) = s(v) + e^r(t+h)\le e^r(1 + t+h)\le e(1 + t+h).
    \]
    By definition of $\xi_\mathcal{M}(g_r\omega)$, the vectors $g_rv$ and $g_rw$ are in $V$ and hence
    \[
    0< v_x,\;w_x<e^{-r/2}\le 1.
    \]
    It follows that $v,\; w\in \Lambda^{\mathrm{red}}_\omega \cap K_{t,h}$. This proves the claim.
    
    We next decompose $\mathcal{E}_{t,h}(\omega)$ as
    \[
    \mathcal{E}_{t,h}(\omega) = 
    \mathcal{E}_{t,h}^0(\omega)
    \bigsqcup 
    \mathcal{E}_{t,h}^1(\omega).
    \]
    The subset $\mathcal{E}_{t,h}^0(\omega)$ consists of $(v,w)\in \mathcal{E}_{t,h}(\omega)$ such that $s(v)<s(w)$ are already consecutive slopes in $\mathbb S_{\omega}$, and $\mathcal{E}_{t,h}^1(\omega)$ is the complement.

    For $\mathcal{E}_{t,h}^0(\omega)$, by consecutiveness, we have
    \[
    \#\mathcal{E}_{t,h}^0(\omega)
    \le
    \#\left(\Lambda^{\mathrm{red}}_\omega \cap K_{t,h}\right).
    \]

    To analyze $\#\mathcal{E}^1_{t,h}(\omega)$, we track along $(g_r\omega)_{0\le r\le 1}$ which slopes disappear and when new pairs become consecutive in $\mathbb{S}_{g_r\omega}$. Define
    \[
    \rho(u) = -2\log u_x
    \]
    for $u\in V$. Hence $\rho(u)$ is the time at which the orbit $(g_ru)_{r\ge 0}$ leaves $V$. In particular, for $u\in\Lambda_\omega^{\mathrm{red}}\cap K_{t,h}$, the slope $s(g_ru)$ belongs to $\mathbb S_{g_r\omega}$ if and only if $0\le r<\rho(u)$.
    
    For $(v,w)\in \mathcal{E}^1_{t,h}(\omega)$, since $s(v)$ and $s(w)$ are not consecutive in $\mathbb S_\omega$, there is at least one slope in $\mathbb S_\omega$ strictly between $s(v)$ and $s(w)$. Set
    \[
     r(v,w)
    :=
    \max\left\{
    \rho(u):
    u\in \Lambda_\omega^{\mathrm{red}}\cap K_{t,h},\;
    s(v)<s(u)<s(w)
    \right\}.
    \]
    Thus $r(v,w)$ is the first time at which $s(g_rv)$ and $s(g_rw)$ become consecutive slopes in $\mathbb S_{g_r\omega}$. Set
    \[
    D(v,w)
    :=
    \left\{
    u\in \Lambda_\omega^{\mathrm{red}}\cap K_{t,h}:
    s(v)<s(u)<s(w),\;
    \rho(u)=r(v,w)
    \right\}.
    \]
    Thus $D(v,w)$ consists of the holonomy vectors that have slope lying strictly between $s(v)$ and $s(w)$, and belong to $\mathbb S_{g_r\omega}$ if and only if $0\le r <  r(v,w)$.
    
    Then define
    \[
    \phi:\mathcal{E}_{t,h}^1(\omega)
    \longrightarrow
    \Lambda^{\mathrm{red}}_\omega \cap K_{t,h}
    \]
    by letting $\phi(v,w)$ be the holonomy vector $u\in D(v,w)$ such that
    \[
    s(u)=\min\{s(u'):u'\in D(v,w)\}.
    \]

    We claim that $\phi$ is injective. Let $u=\phi(v,w)$. Since $u\in D(v,w)$, we have $\rho(u)=r(v,w)$. Among the slopes in $\mathbb S_{g_{\rho(u)}\omega}$, $s(v)$ is the largest one below $s(u)$ and $s(w)$ is the smallest one above $s(u)$. More precisely,
    \[
    s(v)
    =
    \max\left\{
    s(z):
    z\in\Lambda_\omega^{\mathrm{red}}\cap K_{t,h},\;
    \rho(z)>\rho(u),\;
    s(z)<s(u)
    \right\},
    \]
    and
    \[
    s(w)
    =
    \min\left\{
    s(z):
    z\in\Lambda_\omega^{\mathrm{red}}\cap K_{t,h},\;
    \rho(z)>\rho(u),\;
    s(z)>s(u)
    \right\}.
    \]
    Hence $u=\phi(v,w)$ uniquely determines $s(v)$ and $s(w)$.
    Since there is a unique holonomy vector in each direction by definition of $\Lambda_\omega^{\mathrm{red}}$,
    $u=\phi(v,w)$ also uniquely determines $(v,w)$. Therefore $\phi$ is injective, and hence
    \[
    \#\mathcal{E}_{t,h}^1(\omega)
    \le
    \#\left(\Lambda^{\mathrm{red}}_\omega \cap K_{t,h}\right).
    \]
    
    Combining the two estimates, we obtain
    \[
    \#\mathcal{E}_{t,h}(\omega)
    \le
    2\#\left(\Lambda^{\mathrm{red}}_\omega \cap K_{t,h}\right).
    \]
\end{proof}

We now complete the proof of Proposition~\ref{prop:one-dimensional-Lipschitz}.

\begin{proof}
Combining~\eqref{equ:1-step-rewrite-difference}, Lemma~\ref{lem:bound-integral-of-counting} and Lemma~\ref{lem:bound-pair-subset}, we have
\[
\begin{aligned}
G_{\mathcal M;0}(t+h)-G_{\mathcal M;0}(t)
&=
\frac{1}{\lambda}
E_\mathcal M
\left[
\int_0^1 N_{t,h}(g_r\omega)\,dr
\right] \\
&\le
\frac{2}{\lambda}
E_\mathcal M
\left[
\#\left(
\Lambda_\omega^{\mathrm{red}}\cap K_{t,h}
\right)
\right]
\frac{h}{t}\\
&=
\frac{2}{\lambda} 
c^{\mathrm{red}}(\mu)
\operatorname{Leb}(K_{t,h})\frac{h}{t} 
= 
2e\,
(1+t+h)\frac{h}{t},
\end{aligned}
\]
where the second equality follows from the Siegel--Veech formula in
Lemma~\ref{lem:SV-quadratic-short-holonomy-vectors}, and the last equality
uses
\[
\operatorname{Leb}(K_{t,h})=\frac e2(1+t+h)
\]
and $\lambda=c^{\mathrm{red}}(\mu)/2$ from
Lemma~\ref{lem:slope-point-process}.

Now fix $a>0$ and consider $t\ge a$. If $0<h\le t$, then
\[
G_{\mathcal M;0}(t+h)-G_{\mathcal M;0}(t)
\le
2e
\left(2+\frac{1}{a}\right)h.
\]
If $h>t\ge a$, then
\[
G_{\mathcal M;0}(t+h)-G_{\mathcal M;0}(t) 
= 
P^0_\mathcal{M}\left(\Delta_0\zeta\in(t,t+h]\right)
\le
1
<
\frac{h}{a}.
\]
Therefore, we have
\[
0\le
G_{\mathcal M;0}(t')-G_{\mathcal M;0}(t)
\le
2e
\left(2+\frac{1}{a}\right)
(t'-t)
\]
for all $a\le t<t'$. Thus $G_{\mathcal M;0}$ is Lipschitz continuous on $[a,\infty)$ for every $a>0$. Its continuity at $0$ follows from the right continuity of distribution functions.
\end{proof}

We now turn to higher-dimensional distribution functions.

\begin{corollary}\label{cor:joint-gap-local-Lipschitz}
The function $G_{\mathcal M;\,j_1,\ldots,j_m}$ is continuous on $[0,\infty)^m$, and for every $a>0$ and $(t_\ell), (t'_\ell)\in[a,\infty)^m$, 
\[
    \left|G_{\mathcal M;\,j_1,\ldots,j_m}(t_1',\ldots,t_m')
    -
    G_{\mathcal M;\,j_1,\ldots,j_m}(t_1,\ldots,t_m)\right| 
    \le 
    2e
    \left(2+\frac{1}{a}\right) \sum_{\ell = 1}^{m} |t_\ell'-t_\ell|.
\]
\end{corollary}

\begin{proof}
By the inequality from~\cite[Lemma~1.2.16]{DuranteSempi2016}, we obtain
\begin{equation}\label{equ:distribution-differences}
\left| 
G_{\mathcal M;\,j_1,\ldots,j_m}(t_1',\ldots,t_m')
-
G_{\mathcal M;\,j_1,\ldots,j_m}(t_1,\ldots,t_m)\right| 
\le 
\sum_{\ell=1}^m \left| G_{\mathcal M;\,j_\ell}(t_\ell') - G_{\mathcal M;\,j_\ell}(t_\ell) \right|.
\end{equation} 
By~\eqref{equ:interval-stationarity}, we have $G_{\mathcal M;\,j_\ell}=G_{\mathcal M;\,0}$ for every $j_\ell$. 
Since $G_{\mathcal M;\,0}$ is continuous on $[0,\infty)$, the above inequality implies the continuity of $G_{\mathcal M;\,j_1,\ldots,j_m}$ on $[0,\infty)^m$. If moreover, $(t_\ell),(t_\ell')\in[a,\infty)^m$ for some $a>0$, then Proposition~\ref{prop:one-dimensional-Lipschitz} bounds the right-hand side by 
\[
2e
\left(2+\frac{1}{a}\right) \sum_{\ell=1}^m |t_\ell'-t_\ell|.
\]
\end{proof}

\subsection{Absolute continuity of higher-dimensional distribution functions}\label{subsec:gap-ab-c}

We now prove Theorem~\ref{thm:joint-gap-distributions}(3) in the general setting of $\mu$ and $\mathcal M$.

\begin{proposition}\label{prop:absolutely-continuous}
    For any distinct $j_1,\ldots,j_m\in\mathbb Z$, if $m\le 2$, then  $G_{\mathcal M;\, j_1,\ldots,j_m}$ is absolutely continuous. Moreover, if $m\ge \dim_{\mathbb R}\mathcal M$, then $G_{\mathcal M;\, j_1,\ldots,j_m}$  is not absolutely continuous.
\end{proposition}

As defined in the introduction, absolute continuity of $G_{\mathcal M;j_1,\ldots,j_m}$ means that it admits an $L^1$ density. Since 
\[
G_{\mathcal M;j_1,\ldots,j_m}(t_1,\ldots,t_m) = P^0_{{\mathcal M}}\left(\Delta_{j_1}\xi\le t_1,\ldots,\Delta_{j_m}\xi\le t_m\right),
\]
this is equivalent to 
\[
P^0_{{\mathcal M}}
\circ
\left(\Delta_{j_1},\ldots,\Delta_{j_m}\right)^{-1}
\ll
\operatorname{Leb}_m.
\]
We will prove the latter condition directly.

If $m=1$, local Lipschitz continuity of a distribution function implies the absolute continuity. In higher dimensions, this implication fails. We therefore treat the joint distributions separately.

We first study a local map on $\mathcal{M}$ defined by slope differences. Let $\mathcal{H}(\kappa)$ be the ambient stratum of $\mathcal{M}$. Let
\[
    \Phi:U\longrightarrow V\subset H^1(S,\Sigma;\mathbb C)
\]
be a period coordinate chart on $\mathcal{H}(\kappa)$, and identify $U$ with $V$ via $\Phi$. By the definition of an affine invariant submanifold, after shrinking $U$ if necessary, $U\cap\mathbb R_{>0}\mathcal M$ is a finite union of linear subspaces of $H^1(S,\Sigma;\mathbb C)$ intersected with $V$.

For $\omega\in U$ and $\alpha\in H_1(S,\Sigma;\mathbb Z)$, write $\int_\alpha\omega = \left(x_\alpha(\omega),y_\alpha(\omega)\right)$. Whenever $x_\alpha(\omega)\ne0$, write
\[
    s_\alpha(\omega)
    :=
    s\left(\int_\alpha\omega\right)
    =
    \frac{y_\alpha(\omega)}{x_\alpha(\omega)}.
\]
For $\alpha,\beta\in H_1(S,\Sigma;\mathbb Z)$ such that $x_\alpha(\omega),x_\beta(\omega)\ne0$, define
\[
    d_{\alpha,\beta}(\omega)
    := s_\beta(\omega) - s_\alpha(\omega).
\]
For every $\alpha_1,\beta_1,\ldots,\alpha_m,\beta_m\in H_1(S,\Sigma;\mathbb Z)$, define the open subset
\[
    V_{\alpha_1,\beta_1,\ldots,\alpha_m,\beta_m}
    :=
    \left\{
    \omega\in U:
    x_{\alpha_\ell}(\omega)x_{\beta_\ell}(\omega)\ne0
    \text{ for all }1\le\ell\le m
    \right\},
\]
and define
\[
    \boldsymbol d_{\alpha_1,\beta_1,\ldots,\alpha_m,\beta_m}
    :=
    \left(
    d_{\alpha_1,\beta_1},
    \ldots,
    d_{\alpha_m,\beta_m}
    \right)
    :
    V_{\alpha_1,\beta_1,\ldots,\alpha_m,\beta_m}\cap\mathcal{M}
    \longrightarrow\mathbb R^m.
\]

For $m=2$, let $A_{\alpha_1,\beta_1,\alpha_2,\beta_2}\subset V_{\alpha_1,\beta_1,\alpha_2,\beta_2}$ be the set of $\omega$ satisfying
\[
    s_{\alpha_1}(\omega)
    <
    s_{\beta_1}(\omega)
    \le
    s_{\alpha_2}(\omega)
    <
    s_{\beta_2}(\omega).
\]

\begin{lemma}\label{lem:two-slope-differences-submersion}
    Fix any period coordinate chart $(U,\Phi)$ and $\alpha_1,\beta_1,\ldots,\alpha_m,\beta_m\in H_1(S,\Sigma;\mathbb Z)$.
    \begin{enumerate}
        \item For every Borel set $E\subset\mathbb R^2$ with $\operatorname{Leb}_2(E)=0$, the set
        \[
        A_{\alpha_1,\beta_1,\alpha_2,\beta_2}
        \bigcap
        \;(d_{\alpha_1,\beta_1,\alpha_2,\beta_2})^{-1}(E)
        \]
        has $\mu$-measure zero.
        \item If $m\ge \dim_\mathbb R\mathcal{M}$, then the image of $\boldsymbol d_{\alpha_1,\beta_1,\ldots,\alpha_m,\beta_m}$ has Lebesgue measure zero in $\mathbb R^m$.
    \end{enumerate}
\end{lemma}

\begin{proof}
For~(1), since the singular locus of $\mathcal{M}$ has $\mu$-measure zero, it suffices to prove that intersection of the smooth locus with $A_{\alpha_1,\beta_1,\alpha_2,\beta_2} \bigcap\;(d_{\alpha_1,\beta_1,\alpha_2,\beta_2})^{-1}(E)$ has $\mu$-measure zero. Let $\omega$ be a point in this intersection.

We first show that $d_{\alpha_1,\beta_1,\alpha_2,\beta_2}$ is a submersion at $\omega$. Write
\[
    a:=s_{\alpha_1}(\omega),
    \qquad
    b:=s_{\beta_1}(\omega),
    \qquad
    c:=s_{\alpha_2}(\omega),
    \qquad
    d:=s_{\beta_2}(\omega).
\]
Since $\omega\in A_{\alpha_1,\beta_1,\alpha_2,\beta_2}$, we have $a<b\le c<d$. Consider 
$g_t
=
\begin{pmatrix}
e^{t/2} & 0\\
0 & e^{-t/2}
\end{pmatrix}$
and 
$u_r
=
\begin{pmatrix}
1 & r\\
0 & 1
\end{pmatrix}$.
By a direct computation,
\[ 
    \left. \frac{d}{dt} \right|_{t=0} d_{\alpha_1,\beta_1}(g_t\omega) 
    = 
    \left. \frac{d}{dt} \right|_{t=0} \left(e^{-t}b-e^{-t}a\right) 
    = -(b-a) = -d_{\alpha_1,\beta_1}(\omega),
\]
and
\[
    \begin{aligned}
    \left. \frac{d}{dr} \right|_{r=0} d_{\alpha_1,\beta_1}(u_r\omega) 
    &= \left. \frac{d}{dr} \right|_{r=0} \left( \frac{b}{1+rb} - \frac{a}{1+ra} \right)\\ 
    &= -b^2+a^2 = -d_{\alpha_1,\beta_1}(\omega)(a+b),
    \end{aligned}
\]
Similarly,
\[
    \left.
    \frac{d}{dt}
    \right|_{t=0}
    d_{\alpha_2,\beta_2}(g_t\omega)
    =
    -d_{\alpha_2,\beta_2}(\omega)\quad\text{and}\quad
    \left.
    \frac{d}{dr}
    \right|_{r=0}
    d_{\alpha_2,\beta_2}(u_r\omega)
    =
    -d_{\alpha_2,\beta_2}(\omega)(c+d).
\]

Restricted to the tangent subspace spanned by $\left.\frac{d}{dt}\right|_{t=0}g_t\omega$ and $\left.\frac{d}{dt}\right|_{r=0}u_r\omega$, the tangent map $D\boldsymbol d_{\alpha_1,\beta_1,\alpha_2,\beta_2}(\omega)$ is represented by
\[
    -
    \begin{pmatrix}
    d_{\alpha_1,\beta_1}(\omega)
    &
    d_{\alpha_1,\beta_1}(\omega)(a+b)
    \\
    d_{\alpha_2,\beta_2}(\omega)
    &
    d_{\alpha_2,\beta_2}(\omega)(c+d)
    \end{pmatrix}.
\]
Its determinant is $d_{\alpha_1,\beta_1}(\omega)d_{\alpha_2,\beta_2}(\omega)(c+d-a-b)$. Since $a<b\le c<d$, the determinant is positive and hence $\boldsymbol d_{\alpha_1,\beta_1,\alpha_2,\beta_2}$ is a submersion at $\omega$.

Note that $\mu$ belongs to the Lebesgue measure class on $\mathcal M$ and the preimage under a submersion of a set of Lebesgue measure zero has Lebesgue measure zero. The above shows that the preimage $(d_{\alpha_1,\beta_1,\alpha_2,\beta_2})^{-1}(E)$ within  $A_{\alpha_1,\beta_1,\alpha_2,\beta_2}$ has $\mu$-measure zero. Then the first assertion follows.

For~(2), 
let
$h_s=\begin{pmatrix}
1&0\\ 
-s&1
\end{pmatrix}$.
For every $\omega\in V_{\alpha_1,\beta_1,\ldots,\alpha_m,\beta_m}\cap\mathcal M$, $d_{\alpha_\ell,\beta_\ell}(h_s\omega)=d_{\alpha_\ell,\beta_\ell}(\omega)$ for each $\ell$. If $\omega$ is also in the smooth locus of $\mathcal{M}$, then the tangent map $D\boldsymbol d_{\alpha_1,\beta_1,\ldots,\alpha_m,\beta_m}(\omega)$ is defined and the tangent vector $\left.\frac{d}{ds}\right|_{s=0}h_s\omega$ belongs to its kernel. Consequently,
\[
\operatorname{rank}
    D\boldsymbol d_{\alpha_1,\beta_1,\ldots,\alpha_m,\beta_m}(\omega)
    \le
    \dim_{\mathbb R}\mathcal M-1
    <m.
\]
Hence, by Sard's theorem, the image of the smooth locus under $\boldsymbol d_{\alpha_1,\beta_1,\ldots,\alpha_m,\beta_m}$ has Lebesgue measure zero in $\mathbb R^m$. The singular locus is locally contained in a finite union of lower-dimensional linear subspaces in period coordinates, and hence the images also have Lebesgue measure zero by Sard's theorem. Therefore, the entire image of $\boldsymbol d_{\alpha_1,\beta_1,\ldots,\alpha_m,\beta_m}$ has Lebesgue measure zero in $\mathbb R^m$.
\end{proof}

We are ready to prove Proposition~\ref{prop:absolutely-continuous}.

\begin{proof}
Fix an interval $I\subset\mathbb R$ with $\operatorname{Leb}(I)=1$. By the refined Campbell theorem,
\begin{equation}\label{equ:Campbell-higher-dim-distribution}
\begin{aligned}
    &\lambda
    P^0_{{\mathcal M}}
    \big(
    \left(\Delta_{j_1}\zeta,\ldots,\Delta_{j_m}\zeta\right)\in B
    \big)\\
    &\qquad=
    \int_{\mathcal M}
    \sum_{s\in\mathbb S_\omega\cap I}
    \chi_B
    \Big(
    \Delta_{j_1}\big(\tau_s\xi_{\mathcal M}(\omega)\big),
    \ldots,
    \Delta_{j_m}\big(\tau_s\xi_{\mathcal M}(\omega)\big)
    \Big)
    \,d\mu(\omega)
\end{aligned}
\end{equation}
for every $B\in\mathcal B(\mathbb R^m)$. 

Fix a period coordinate chart $(U,\Phi)$. For every $\omega\in U\cap\mathcal M$ and $s\in\mathbb S_\omega\cap I$, there are relative homology classes $\alpha_\ell,\beta_\ell$, $1\le\ell\le m$, represented by saddle connections contributing to $\Delta_{j_\ell}(\tau_s\xi_{\mathcal M}(\omega))$, such that
\[
    \Delta_{j_\ell}(\tau_s\xi_{\mathcal M}(\omega))
    =
    d_{\alpha_\ell,\beta_\ell}(\omega).
\]
Then
\begin{equation}\label{equ:equal-slope-difference}
    \Big(
    \Delta_{j_1}\big(\tau_s\xi_{\mathcal M}(\omega)\big),\ldots,
    \Delta_{j_m}\big(\tau_s\xi_{\mathcal M}(\omega)\big)
    \Big)
    =
    \boldsymbol d_{\alpha_1,\beta_1,\ldots,\alpha_m,\beta_m}(\omega).
\end{equation}

We first prove the conclusion for $m\le2$. For $m=1$, local Lipschitz continuity already implies absolute continuity, so it remains to consider $m=2$. Let $E\subset(0,\infty)^2$ be any Borel set with $\operatorname{Leb}_2(E)=0$. We aim to show that
\[
    P^0_{{\mathcal M}}
    \big(
    \left(\Delta_{j_1}\zeta,\Delta_{j_2}\zeta\right)\in E
    \big)
    =0.
\]

We may assume that $j_1<j_2$. For $\omega\in U\cap\mathcal{M}$, fix the classes $\alpha_1,\beta_1,\alpha_2,\beta_2$ appearing above. Since $j_1<j_2$,  we have
\[
s_{\alpha_1}(\omega) < s_{\beta_1}(\omega)
\le s_{\alpha_2}(\omega) < s_{\beta_2}(\omega)
\]
and hence $\omega\in A_{\alpha_1,\beta_1,\alpha_2,\beta_2}\cap\mathcal{M}$.  Therefore, whenever $\chi_E
\Big(
\Delta_{j_1}\big(\tau_s\xi_{\mathcal M}(\omega)\big),
\Delta_{j_2}\big(\tau_s\xi_{\mathcal M}(\omega)\big)
\Big)$ 
is positive, $\omega$ is contained in
\[
    \bigcup_{\alpha_1,\beta_1,\alpha_2,\beta_2\in H_1(S,\Sigma;\mathbb Z)}
    \left(A_{\alpha_1,\beta_1,\alpha_2,\beta_2}\bigcap\;
    (\boldsymbol d_{\alpha_1,\beta_1,\alpha_2,\beta_2})^{-1}(E)\right).
\]
By Lemma~\ref{lem:two-slope-differences-submersion}~(1), each set in the union has $\mu$-measure zero. Since $H_1(S,\Sigma;\mathbb Z)$ is countable, the union also has $\mu$-measure zero. Thus, $\chi_E
    \Big(
    \Delta_{j_1}\big(\tau_s\xi_{\mathcal M}(\omega)\big),
    \Delta_{j_2}\big(\tau_s\xi_{\mathcal M}(\omega)\big)
    \Big)$
vanishes for $\mu$-almost every $\omega\in U\cap\mathcal M$. 

Taking $B:= E$ in~\eqref{equ:Campbell-higher-dim-distribution} and taking a countable cover of $\mathcal{M}$ by period coordinate domains, the right-hand side is zero in each domain. This proves $P^0_{{\mathcal M}}\big( \left(\Delta_{j_1}\zeta,\Delta_{j_2}\zeta\right)\in E \big) = 0$ and hence the absolute continuity for $m=2$.

We next prove the conclusion for $m\ge \dim_\mathbb R\mathcal{M}$. We aim to construct a subset $B_0\subset\mathbb R^m$ with $\mathrm{Leb}_m(B_0)=0$ such that
\[
P^0_{\mathcal M}\left(
\left(\Delta_{j_1}\zeta,\ldots,\Delta_{j_m}\zeta\right)\in B_0
\right)>0.
\]

For each period coordinate chart $(U,\Phi)$, take
\[
    B_U:=
    \bigcup_{\alpha_1,\beta_1,\ldots,\alpha_m,\beta_m\in H_1(S,\Sigma;\mathbb Z)}
    \boldsymbol d_{\alpha_1,\beta_1,\ldots,\alpha_m,\beta_m}
    \left(
    V_{\alpha_1,\beta_1,\ldots,\alpha_m,\beta_m}\cap\mathcal{M}
    \right)\subset \mathbb R^m.
\]
By Lemma~\ref{lem:two-slope-differences-submersion}~(2), each subset in the union has Lebesgue measure zero, and hence $B_U$ has Lebesgue measure zero.

Take a countable cover $(U_n)_{n\in\mathbb N}$ of $\mathcal{M}$ by period coordinate domains, and choose a Borel set $B_0\subset \mathbb R^m$ of Lebesgue measure zero containing $\bigcup_{n\in\mathbb N} B_{U_n}$.
Then by~\eqref{equ:equal-slope-difference}, 
\[
    \left(
    \Delta_{j_1}(\tau_s\xi_{\mathcal M}(\omega)),\ldots,
    \Delta_{j_m}(\tau_s\xi_{\mathcal M}(\omega))
    \right)\in B_0
\]
for every $\omega\in\mathcal M$ and $s\in\mathbb S_\omega\cap I$. By~\eqref{equ:Campbell-higher-dim-distribution},
\[
    P^0_{\mathcal M}
    \big(
        \left(\Delta_{j_1}\zeta,\ldots,\Delta_{j_m}\zeta\right)\in B_0
    \big)
    =1.
\]
Since $\mathrm{Leb}_m(B_0) = 0$,  $G_{\mathcal M;\, j_1,\ldots,j_m}$ is not absolutely continuous.
\end{proof}

\subsection{Proof of Theorem~\ref{thm:joint-gap-distributions} and Theorem~\ref{thm:general-joint-gap-distributions-in-intro}}

We now prove full version of Theorem~\ref{thm:general-joint-gap-distributions-in-intro}.

\begin{theorem}\label{thm:general-joint-gap-distributions}
Assume that $\mu$ satisfies~\ref{ass:circle-limit}. Then for every $m\in\mathbb N$ and any distinct $j_1,\ldots,j_m\in\mathbb Z$, the following statements hold.
\begin{enumerate}
    \item
    For $\mu$-almost every $\omega\in\mathcal M$ and every
    $t_1,\ldots,t_m\ge0$,
    \[
    \begin{aligned}
    &\lim_{R\to\infty}
    \frac{1}{D_R(\omega)}
    \#\left\{
    1\le i\le D_R(\omega):
    R^2\Delta^\omega_{i+j_\ell}\le t_\ell
    \text{ for }1\le\ell\le m
    \right\}
    \\
    &\hspace{4cm}
    =
    G_{\mathcal M;\,j_1,\ldots,j_m}
    (t_1,\ldots,t_m).
    \end{aligned}
    \]

    \item The function $G_{\mathcal M;\,j_1,\ldots,j_m}$ is continuous on $[0,\infty)^m$ and Lipschitz continuous on $[a,\infty)^m$ for every $a>0$.

    \item
    The function $G_{\mathcal M;\,j_1,\ldots,j_m}$ is absolutely continuous for $m\le2$, but is not absolutely continuous for $m\ge\dim_{\mathbb R}\mathcal M$.
\end{enumerate}
\end{theorem}

\begin{proof}
Assertion~(2) follows from Corollary~\ref{cor:joint-gap-local-Lipschitz}. Assertion~(1) follows from Proposition~\ref{prop:convergence-joint-gap-distributions} and the continuity in assertion~(2). Assertion~(3) follows from Proposition~\ref{prop:one-dimensional-Lipschitz} and Proposition~\ref{prop:absolutely-continuous}.
\end{proof}

We next return to Theorem~\ref{thm:joint-gap-distributions}. We specialize to the Masur--Veech measure $\mu_\mathcal C$ on a connected component $\mathcal C$ of a stratum $\mathcal H^1(\kappa)$. It remains to prove the strict positivity assertion in Theorem~\ref{thm:joint-gap-distributions}~(2). The following proof generalizes the positivity argument in~\cite[Section~3.2]{AthreyaChaika2012} for one-dimensional gap distributions. Let
\[
    \widetilde p_k(\sigma) 
    := 
    \mu_\mathcal{C}\left\{\omega:\#(\Lambda_\omega\cap T_{(-\sigma,\sigma)}) \ge k\right\},\qquad k\in\mathbb N,\ \sigma>0,
\]
where $T_{(-\sigma,\sigma)}$ is a triangle as defined in~\eqref{equ:triangle-in-plane}. By~\cite[Lemma~5.4]{AthreyaChaika2012}, $\widetilde p_k(\sigma)>0$ for every $k\in\mathbb N$ and $\sigma>0$.

\begin{lemma}\label{lem:positivity-gap-distributions}
    Fix a connected component $\mathcal{C}$ of $\mathcal{H}^1(\kappa)$. For every $m\in\mathbb N$ and $j_1,\ldots,j_m\in\mathbb Z$, the distribution function $G_{\mathcal{C};j_1,\ldots,j_m}$ is strictly positive on $(0,\infty)^m$.
\end{lemma}

\begin{proof}
    Fix $(t_1,\ldots,t_m)\in(0,\infty)^m$ and take $I=(-1/2,1/2)$ in~\eqref{equ:rephrase-slope-gap-distribution}. Set
    \[
        k:=\max_{1\le \ell\le m}|j_\ell|,
        \qquad
        \sigma:=\frac12\min\{1,t_1,\ldots,t_m\},
    \]
    and let $J:=(-\sigma,\sigma) \subset I$. Consider
    \[
        B :=
        \left\{
        \omega\in\mathcal C:
        \#\left(\Lambda_\omega\cap T_J\right)\ge 2k+2
        \right\}.
    \]
    Then $\mu_{\mathcal C}(B) = \widetilde p_{2k+2}(\sigma)>0$.

    By Lemma~\ref{lem:SV-quadratic-short-holonomy-vectors}, $\Lambda_\omega=\Lambda_\omega^{\mathrm{red}}$ for $\mu_{\mathcal C}$-almost every $\omega$. Hence, for $\mu_{\mathcal C}$-almost every $\omega\in B$, the slope point process $\xi_{\mathcal C}(\omega)$ has at least $2k+2$ points in $J$. As $J$ is an interval, we may choose $2k+2$ consecutive slopes in $J$
    \[
    s_p(\omega)<s_{p+1}(\omega)<\cdots<s_{p+2k+1}(\omega).
    \]
     
    Set $i_0:=p+k$. Since $|j_\ell|\le k$, we have $p\le i_0+j_\ell<i_0+j_\ell+1\le p+2k+1$ and hence $s_{i_0+j_\ell},s_{i_0+j_\ell+1}\in J$ for each $1\le\ell\le m$. In particular,
    \[
        s_{i_0+j_\ell+1}(\omega)-s_{i_0+j_\ell}(\omega)
        \le
        \mathrm{Leb}(J)
        =
        2\sigma
        \le
        t_\ell.
    \]
    It follows that 
    \[
    s_{i_0}(\omega) \in \left\{s_i(\omega)\in I: s_{i+j_\ell+1}(\omega) - s_{i+j_\ell}(\omega) \le t_\ell \text{ for }1\le\ell\le m \right\}.
    \]
    Then the integrand in~\eqref{equ:rephrase-slope-gap-distribution} is at least one for $\mu_{\mathcal C}$-almost every $\omega\in B$, and
    \[
    \begin{aligned}
        G_{\mathcal C;j_1,\ldots,j_m}(t_1,\ldots,t_m)
        &= 
        \frac{1}{\lambda_{\mathcal C}} 
        \int_{\mathcal C} 
        \#\big\{s_i(\omega)\in I: 
        s_{i+j_\ell+1}(\omega) - s_{i+j_\ell}(\omega) 
        \le t_\ell \text{ for }1\le\ell\le m \big\}
        \,d\mu(\omega)\\
        &\ge
        \frac{\mu_{\mathcal C}(B)}{\lambda_{\mathcal C}}
        >
        0.
    \end{aligned}
    \]
\end{proof}

\begin{proof}[Proof of Theorem~\ref{thm:joint-gap-distributions}]
    When $\mu$ is the Masur--Veech measure $\mu_{\mathcal C}$ on a connected component $\mathcal C$ of $\mathcal H^1(\kappa)$, assumption~\ref{ass:circle-limit} follows from a result of Nevo; see for instance~\cite{EskinMasur2001}. Hence, all the assertions of Theorem~\ref{thm:general-joint-gap-distributions} apply to $\mu_{\mathcal C}$. The remaining strict positivity assertion in Theorem~\ref{thm:joint-gap-distributions}(2) follows from Lemma~\ref{lem:positivity-gap-distributions}.
\end{proof}

\begin{remark}[Veech surfaces]\label{rmk:Veech-gap-distributions}
Assume that $\mathcal M$ is the closed $\mathrm{SL}(2,\mathbb R)$-orbit of a Veech surface. Then assumption~\ref{ass:circle-limit} is satisfied. See e.g. Proposition 2.2.1 of Kleinbock-Margulis \cite{MR1359098}.

In this case, the conclusions of Theorem~\ref{thm:general-joint-gap-distributions} can be sharpened. By~\cite[Lemma~3.6]{AthreyaChaika2012} and Proposition~\ref{prop:convergence-joint-gap-distributions}, $G_{\mathcal M;\,j_1,\ldots,j_m}$ vanishes in a neighborhood of the origin. Together with~(2), we obtain $G_{\mathcal M;j_1,\ldots,j_m}$ is Lipschitz on all of $[0,\infty)^m$.

Finally, $\dim_{\mathbb R}\mathcal M=3$. Thus~(3) completely determines the absolute continuity of the joint gap distributions for every $m\in\mathbb N$.
\end{remark}

\begin{remark}[Connection to~\cite{AthreyaChaika2012}]\label{rmk:connection-AC}
    Assume that $\mu$ satisfies~\ref{ass:circle-limit}. Athreya and Chaika~\cite{AthreyaChaika2012} use~\cite[Theorem~2.2]{Marklof2007} to relate the gap distribution $G_\mathcal{M}$ to
\[
    p_0(\sigma) := \mu
    \left\{
    \omega:
    \Lambda_\omega\cap T_{[-\sigma,\sigma]}=\varnothing
    \right\},
    \qquad \sigma\geq 0.
\]
    Indeed, according to~\cite[Theorem~2.2]{Marklof2007}, one has
    \begin{equation}\label{equ:tail-GC-and-p0-derivative}
G_{\mathcal M}(t)= 1+ \frac{1}{c^{\mathrm{red}}(\mu)}p_0'\left(\frac{t}{2}\right).
    \end{equation}
This identity can also be derived from the Palm inversion formula~\eqref{equ:Palm-inversion}. Indeed, taking $f:= \chi_{\{\zeta([-\sigma,\sigma])=0\}}$ in~\eqref{equ:Palm-inversion}, we obtain
\[
\begin{aligned}
    p_0(\sigma) 
    &= 
    E_{\mathcal{M}}[\chi_{\{\zeta([-\sigma,\sigma])=0\}}]\\ 
    &= 
    \frac{c^{\mathrm{red}}(\mu)}{2}
    \int_0^\infty E^0_\mathcal{M}
    \left[
    \chi_{\{\zeta([-\sigma,\sigma])=0\}}(\tau_s\zeta_\mathcal{M})
    \chi_{\{\Delta_0\zeta>s\}}(\zeta_\mathcal{M})
    \right]
    ds.
\end{aligned}
\]
Note that $\tau_s(\zeta)([-\sigma,\sigma]) = 0$ if and only if $s>\sigma$ and $\Delta_0\zeta > \sigma + s$. Setting $t = \sigma + s$, we obtain
\begin{equation}\label{equ:tail-GC-and-p0}
     p_0(\sigma) 
    =
    \frac{c^{\mathrm{red}}(\mu)}{2}
    \int_{2\sigma}^\infty P^0_\mathcal{M}
    \left(
    \Delta_0\zeta>t
    \right)
    dt
    =
    \frac{c^{\mathrm{red}}(\mu)}{2}
    \int_{2\sigma}^\infty
    \left(
    1 - G_{\mathcal{M};0}(t)
    \right)
    dt.
\end{equation}
Then differentiating it gives~\eqref{equ:tail-GC-and-p0-derivative}.

In addition, the integral in~\eqref{equ:tail-GC-and-p0} implies that
\[
t(1-G_{\mathcal M}(2t)) \leq \int_t^{2t}(1-G_{\mathcal M}(s))\,ds\to0,
\]
as $t\to\infty$. Then we obtain a tail estimate $1-G_{\mathcal M}(t)=o(t^{-1})$ as $t\to \infty$. Stronger quadratic tail estimates are known for Veech surfaces and certain branched covers of tori; see~\cite[Theorem~4.4\textup{(4)} and Corollary~4.6]{Sanchez2022} and~\cite[Theorem~2]{KumanduriSanchezWang2024}.
\end{remark}

\subsection{Connection to Poincar\'e sections for the horocycle flow}\label{subsec:connection-to-Poincare-section}

We relate the Palm distribution $P^0_\mathcal{M}$ to the horocycle flow on $\mathcal{M}$.

We recall the construction of a Poincar\'e section for the horocycle flow on $\mathcal{M}$. This construction is described in generality by Athreya~\cite{Athreya2016Gap}, and here we mainly follow the description from Section 2 of \cite{BS26}.
Define
\[
\mathcal{W}_\mathcal{M} := \left\{\omega\in \mathcal{M}: \Lambda_\omega\cap (0,1]\times\{0\}\ne \varnothing\right\}.
\]
They show that for $\mu$-almost every $\omega\in\mathcal{M}$, the orbit $(h_s\omega)_{s\ge 0}$ intersects $\mathcal{W}_\mathcal{M}$, and hence $\mathcal{W}_\mathcal{M}$ is a Poincar\'e section for the horocycle flow on $\mathcal{M}$.

The return time function is defined by
\[
R_\mathcal{M}: \mathcal{W}_\mathcal{M} \to (0,\infty),
\qquad 
R_\mathcal{M}(\omega) := \min\{s\in(0,\infty): h_s\omega\in\mathcal{W}_\mathcal{M}\},
\]
and the first-return map is defined by
\[
T_\mathcal{M}: \mathcal{W}_\mathcal{M}\to \mathcal{W}_\mathcal{M},
\qquad
T_\mathcal{M}(\omega) = h_{R_\mathcal{M}(\omega)}\omega.
\]

One can disintegrate $\mu$ on $\mathcal{M}$ to obtain a measure $\boldsymbol{m}_\mathcal{M}$ on $\mathcal{W}_\mathcal{M}$; see~\cite{BS26} for details. The measure $\boldsymbol{m}_\mathcal{M}$ is invariant under $T_\mathcal{M}$, and
\begin{equation}\label{equ:disintegration}
    \int_\mathcal{M}f\,d\mu = C\int_{\mathcal{W}_\mathcal{M}}\int_0^{R_\mathcal{M}(\omega)}f(h_s\omega)\,ds\,d\boldsymbol{m}_\mathcal{M}(\omega)
\end{equation}
for some constant $C>0$ and every nonnegative measurable function $f$ on $\mathcal{M}$. 

The dynamical system $(\mathcal{W}_\mathcal{M},\boldsymbol{m}_\mathcal{M}, T_\mathcal{M})$ has played a central role in the study of gap distributions and is the main mechanism in the works on gap distributions of translation surfaces cited in the introduction \cite{Al25,AthreyaChaikaLelievre2015,AthreyaCheung2014,BMMUW2023,BS26,KumanduriSanchezWang2024,MY25,OsmanSoutherlandWang2025,Sanchez2022,UyanikWork2016,Work2020} 

We now relate the Palm distribution $P^0_\mathcal{M}$ to $\boldsymbol{m}_\mathcal{M}$. Recall that the slope point process defines a map $\xi_\mathcal{M}:\mathcal{M}\longrightarrow\mathcal{N}_s(\mathbb{R})$, and that $P_\mathcal{M}$ is the pushforward of $\mu$ under $\xi_\mathcal{M}$. Since the slopes in $\xi_\mathcal{M}(\omega)$ form a bi-infinite sequence in $\mathbb R$ for every $\omega\in\mathcal M$, we may regard $\xi_\mathcal{M}$ as a map
\[
    \xi_\mathcal{M}:\mathcal{M}\longrightarrow\mathcal{N}_{s,\mathrm{bi}}(\mathbb{R}),
\]
where $\mathcal{N}_{s,\mathrm{bi}}(\mathbb{R})$ denotes the space of bi-infinite counting measures; see~\eqref{equ:bi-infinite}.

Restricted to $\mathcal{W}_\mathcal{M}$, the map $\xi_\mathcal{M}$ takes values in $\mathcal{N}^0_{s,\mathrm{bi}}(\mathbb R)$. Define the point shift
\[
    \sigma:\mathcal{N}^0_{s,\mathrm{bi}}(\mathbb R)
    \longrightarrow
    \mathcal{N}^0_{s,\mathrm{bi}}(\mathbb R),
    \qquad
    \sigma(\zeta):=\tau_{x_1(\zeta)}\zeta,
\]
which moves $x_1(\zeta)$ to zero. Since $x_0(\zeta) = 0$, we have $x_1(\zeta) = \Delta_0\zeta$. The Palm measure has $P^0_\mathcal{M}(\mathcal{N}^0_{s,\mathrm{bi}}(\mathbb R)) = 1$ by~\eqref{equ:bi-infinite-concentration}.

\begin{proposition}\label{prop:measure-preserving}
    The map
    \[
    \xi_\mathcal{M}: 
    \left(\mathcal{W}_\mathcal{M},
    \frac{1}{\boldsymbol{m}_\mathcal{M}\left(\mathcal{W}_\mathcal{M}\right)}\boldsymbol{m}_\mathcal{M}, 
    T_\mathcal{M}\right)
    \longrightarrow
    \left(\mathcal{N}_{s,\mathrm{bi}}^0(\mathbb{R}),P^0_\mathcal{M},\sigma\right)
    \]
    is a measure-preserving factor map.
\end{proposition}

The argument is essentially an application of the uniqueness conclusion in Theorem~\ref{thm:inversion-formula}.

\begin{proof}
For $\omega\in\mathcal{W}_\mathcal{M}$, the definition of the return time gives
\[
R_\mathcal{M}(\omega) = x_1(\xi_\mathcal{M}(\omega)) = (\Delta_0\circ\xi_\mathcal{M})(\omega).
\]
Hence,
\[
(\xi_\mathcal{M}\circ T_\mathcal{M}) (\omega) 
= \xi_\mathcal{M}(h_{R_\mathcal{M}(\omega)}\omega) 
= \tau_{x_1(\xi_\mathcal{M}(\omega))}(\xi_\mathcal{M}(\omega)) 
= (\sigma\circ \xi_\mathcal{M}) (\omega).
\]
Thus $\xi_\mathcal{M}$ intertwines $T_\mathcal{M}$ and $\sigma$. It remains to show that 
\[
(\xi_\mathcal{M})_\ast \frac{1}{\boldsymbol{m}_\mathcal{M}\left(\mathcal{W}_\mathcal{M}\right)}\boldsymbol{m}_\mathcal{M} = P^0_{\mathcal{M}}.
\]

Let
\[
    Q_\mathcal{M}
    :=
    C\,
    \bigl(\xi_\mathcal{M}|_{\mathcal{W}_\mathcal{M}}\bigr)_*
    \boldsymbol{m}_\mathcal{M}.
\]
For every nonnegative measurable function $f:\mathcal{N}_{s,\mathrm{bi}}(\mathbb R)\to[0,\infty]$, the disintegration formula gives
\begin{equation}\label{equ:section-inversion}
    \begin{aligned}
    \int_{\mathcal{N}_{s,\mathrm{bi}}(\mathbb R)}
    f(\zeta)\,dP_\mathcal{M}(\zeta)
    &=
    \int_\mathcal{M}
    f(\xi_\mathcal{M}(\omega))\,d\mu(\omega)\\
    &=
    C\int_{\mathcal{W}_\mathcal{M}}
    \int_0^{R_\mathcal{M}(\omega)}
    f(\xi_\mathcal{M}(h_s\omega))\,ds\,
    d\boldsymbol{m}_\mathcal{M}(\omega)\\
    &=
    \int_{\mathcal{N}_{s,\mathrm{bi}}^0(\mathbb R)}
    \int_0^{\Delta_0(\zeta)}
    f(\tau_s\zeta)\,ds\,dQ_\mathcal{M}(\zeta),
    \end{aligned}
\end{equation}
where $\tau_s$ is defined in~\eqref{equ:translate-notation}. The last equality uses $\xi_\mathcal{M}(h_s\omega) = \tau_s\xi_\mathcal{M}(\omega)$ and $R_\mathcal{M} = \Delta_0\circ\xi_\mathcal{M}$.

According to Theorem~\ref{thm:inversion-formula}, the Palm distribution is the unique probability measure satisfying the Palm inversion formula. Comparing~\eqref{equ:section-inversion} with~\eqref{equ:Palm-inversion}, we obtain
\[
    P_\mathcal{M}^0
    =
    \frac{1}{\lambda_\mathcal{M}}Q_\mathcal{M}
    =
    \frac{C}{\lambda_\mathcal{M}}
    \bigl(\xi_\mathcal{M}|_{\mathcal{W}_\mathcal{M}}\bigr)_*
    \boldsymbol{m}_\mathcal{M}.
\]
Consequently,
\[
    \bigl(\xi_\mathcal{M}|_{\mathcal{W}_\mathcal{M}}\bigr)_*
    \left(
    \frac{\boldsymbol{m}_\mathcal{M}}
    {\boldsymbol{m}_\mathcal{M}(\mathcal{W}_\mathcal{M})}
    \right)
    =
    P_\mathcal{M}^0.
\]
\end{proof}

\section{A generalization of the Masur--Rafi--Randecker Poisson limit}

In~\cite{masur2025lengthssaddleconnectionsrandom}, Masur--Rafi--Randecker formulate their result as a joint Poisson limit for the numbers of saddle connections in disjoint rescaled length intervals. Their proof can also be generalized to obtain Poisson convergence for holonomy vectors. We explain this generalization in this section. 

For simplicity, we write $\mathcal H_g:=\mathcal H^1(1^{2g-2})$ for the area-one locus of the principal stratum and $\mu_g:=\mu_{\mathcal H_g}$ for its normalized Masur--Veech measure.

\subsection{Poisson convergence of holonomy vectors}

For $\omega\in\mathcal H_g$, recall that $\Lambda_\omega\subset\mathbb R^2$ is the set of holonomy vectors of oriented saddle connections on $\omega$. In particular, $\Lambda_\omega=-\Lambda_\omega$.

Let
\[
    H
    :=
    \{(x,y)\in\mathbb R^2:x>0\}
    \cup
    \{(0,y)\in\mathbb R^2:y>0\},
\]
and define
\[
    \Lambda_\omega^+:=\Lambda_\omega\cap H.
\]
Thus, $\Lambda_\omega^+$ contains exactly one of $v$ and $-v$ for each saddle connection.

We identify $\Lambda_\omega^+$ with the locally finite counting measure $\sum_{v\in\Lambda_\omega^+}\delta_v$ on $H$. Let $\omega$ be distributed according to $\mu_g$. We define a point process $\Lambda^+_g$ by
\[
\Lambda^+_g:\mathcal{H}_g\to\mathcal{N}_s(H),\qquad\Lambda^+_g(\omega):=\Lambda^+_\omega.
\]
As before, we write $g\Lambda_g^+$ for the rescaled point process corresponding to the subset $\{gv:v\in\Lambda_\omega^+\}$. Thus, for a bounded Borel set $B\subset H$, $(g\Lambda_g^+)(\omega)(B)$ counts the points of $\Lambda_\omega^+$ lying in $g^{-1}B:=\{g^{-1}v:v\in B\}$. As before, we write
\[
    (g\Lambda_g^+)(B):
    \omega\longmapsto (g\Lambda_g^+)(\omega)(B)
\]
for the corresponding counting random variable.

The main result of~\cite{masur2025lengthssaddleconnectionsrandom} is formulated as a joint Poisson limit for the numbers of saddle connections in disjoint rescaled length intervals. It can be expressed in our notation as follows.

For $0<a<b$, set
\[
A_{a,b}:=\{v\in H:a\leq |v|\leq b\}.
\]

\begin{theorem}[{\cite[Theorem~1.1]{masur2025lengthssaddleconnectionsrandom}}]
\label{thm:mrr-poisson-limit}
Let $[a_1,b_1],\ldots,[a_k,b_k]$ be pairwise disjoint bounded intervals in $\mathbb R_{>0}$. As $g\to\infty$, the random vector
\[
\left(
(g\Lambda_g^+)(A_{a_1,b_1}), \ldots, (g\Lambda_g^+)(A_{a_k,b_k})
\right)
\]
converges in distribution to a vector of independent Poisson random variables with respective means
\[
8\pi(b_i^2-a_i^2),\qquad 1\leq i\leq k.
\]
\end{theorem}

We refine this result by allowing arbitrary bounded Borel subsets of $H$.

\begin{theorem}\label{thm:holonomy-vectors-ppp}
    As $g\to\infty$, the point process $g\Lambda^+_g$ converges in distribution to the Poisson point process $\Pi_{16}^H$ on $H$ with intensity $16$.

    More precisely, for any pairwise disjoint bounded Borel subsets $B_1,\ldots,B_k$ in $H$, as $g\to\infty$, the vector of random variables
    \[
        \left((g\Lambda^+_g)(B_1),\ldots, (g\Lambda^+_g)(B_k)\right)
    \]
    converge jointly to a vector of independent Poisson random variables with means 
    \[
        \lambda_i = 16\,\mathrm{Leb}(B_i),\qquad 1\le i\le k.
    \]
    That is, for any integers $r_1,\ldots,r_k\in \mathbb Z_{\ge 0}$, 
    \begin{equation}\label{equ:holonomy_vectors-ppp}
        \lim_{g\to \infty}\mu_g\left(\omega: (g\Lambda^+_g)(B_1)=r_1,\,\ldots,\, (g\Lambda^+_g)(B_k)=r_k\right) = \prod_{i=1}^k\frac{\lambda_i^{r_i}e^{-\lambda_i}}{r_i!}.
    \end{equation}
\end{theorem}

The intensity $16$ in Theorem~\ref{thm:holonomy-vectors-ppp} is consistent with the coefficient in Theorem~\ref{thm:mrr-poisson-limit}. Indeed, 
\[ 
16\operatorname{Leb}(A_{a,b}\cap H) = 8\pi(b^2-a^2). 
\] 
Thus, $\Pi_{16}^H(A_{a,b})$ has the mean appearing in Theorem~\ref{thm:mrr-poisson-limit}.

We follow the proof strategy of~\cite{masur2025lengthssaddleconnectionsrandom} and explain in the following subsections the modifications needed in our setting. The main change is that the annuli considered there are replaced by arbitrary bounded Borel sets $g^{-1}B_i\subset H$, so that we record the holonomy vectors themselves rather than only their lengths. Second, since $B_i\subset H$ distinguishes the two orientations of each saddle connection, the factor $2^{-K}$ (where we define $K$ in the course of the proof) appearing in their moment computation is absent.

\subsection{The generic subset}

Fix pairwise disjoint bounded Borel subsets $B_1,\ldots,B_k\subset H$. Choose $b>0$ such that 
\[
B_i\subset B(0,b)
\]
for every $i=1,\ldots,k$, where $B(0,b)$ denotes the metric ball of radius $b$ centered at the origin. Let $\mathcal H_{\mathrm{Gen}}$ and $\mathcal H_{\mathrm{NG}}$ be the generic and non-generic subsets of $\mathcal H_g$ associated with $b$ as in~\cite[Definition~6.1]{masur2025lengthssaddleconnectionsrandom}. In particular, $\mathcal H_{\mathrm{Gen}}$ excludes short loops and pairs of short saddle connections sharing a zero in the sense of that definition.

We first compare the counting of saddle connections with the counting of their holonomy vectors. For a bounded subset $B$ in $H$, we define  
\[
N^{\mathrm{sc}}(\omega, B) :=\#\{\gamma:\text{ saddle connections with holonomy vector in $B$}\}.
\]
Since $B\subset H$, at most one orientation of a saddle connection can have its holonomy vector in $B$.

\begin{lemma}\label{lem:two-countings-agree}
    Fix a bounded Borel subset $B\subset H$ contained in $B(0,b)$ for the fixed $b>0$ above. For $\mu_g$-almost every $\omega\in\mathcal{H}_{\mathrm{Gen}}$, we have
    \[
    N^{\mathrm{sc}}\left(\omega, g^{-1}B\right) = (g\Lambda_g^+)(B).
    \]
\end{lemma}

\begin{proof}
    As a counting measure, $(g\Lambda_g^+)(B)$ counts the number of holonomy vectors in $\Lambda_\omega \cap g^{-1} B$. By~\cite[Proposition~3.1]{EskinMasurZorich2003}, for $\mu_g$-almost every $\omega\in\mathcal{H}_g$, nonhomologous saddle connections have distinct holonomy vectors. Two homologous saddle connections either have the same pair of zeros as endpoints or are both loops. Such configurations of saddle connections of length less than $b/g$ are excluded by the definition of $\mathcal H_{\mathrm{Gen}}$. Hence the saddle connections counted by $N^{\mathrm{sc}}\left(\omega,g^{-1} B\right)$ have distinct holonomy vectors, which proves the claim.
\end{proof}

\subsection{Simultaneous collapsing}

Fix $r_1,\ldots,r_k\in\mathbb Z_{\ge0}$. Following the notation in~\cite{masur2025lengthssaddleconnectionsrandom}, $\widehat{\mathcal{H}}$ is a set of tuples $(X,\omega,\Gamma)$, where $(X,\omega)\in\mathcal{H}_{\mathrm{Gen}}$,  $\Gamma = (\Gamma_1,\ldots,\Gamma_k)$ and $\Gamma_i = (\gamma_{i,1},\ldots, \gamma_{i,r_i})$ is an ordered list of disjoint oriented saddle connections with lengths in $[0,{b}/{g}]$ for every $1\le i\le k$. We write $\widehat{\mu}_g$ for the induced measure on $\widehat{\mathcal H}$. 

Write $\boldsymbol{B} = (B_1,\ldots,B_k)$. Let $\widehat{\mathcal H}_{g^{-1}\boldsymbol B}$ be the subset of $\widehat{\mathcal H}$ for which the holonomy vector of every $\gamma_{i,j}$ belongs to $g^{-1}B_i$.

Set $K = \sum_{1\le i\le k}r_i$. According to~\cite[Section~7]{masur2025lengthssaddleconnectionsrandom}, simultaneously collapsing the pairs of zeros joined by the saddle connections in $\Gamma$ gives a map
\[
    \Phi: \widehat{\mathcal{H}}\to M_{r_1,\ldots,r_k}\times \mathcal{H'} \times \prod_{i=1}^k B(0, b/g)^{r_i}.
\]
Here, $\mathcal H'$ is the stratum obtained by collapsing the $K$ saddle connections and $M_{r_1,\ldots,r_k}$ records their labeled endpoints. Moreover, by~\cite[Section~7]{masur2025lengthssaddleconnectionsrandom}, the map $\Phi$ is $3^K$-to-1 and locally measure preserving.

We now restrict $\Phi$ to $\widehat{\mathcal{H}}_{g^{-1}\boldsymbol{B}}$. Then $\widehat{\mathcal{H}}_{g^{-1}\boldsymbol{B}}$ has image in 
\[
    M_{r_1,\ldots,r_k}\times \mathcal{H'} \times \prod_{i=1}^k (g^{-1}B_i)^{r_i}
\]
under $\Phi$. The restriction $\Phi|_{\widehat{\mathcal{H}}_{g^{-1}\boldsymbol{B}}}$ is also $3^K$-to-$1$. Indeed, if the prescribed vectors $v_{i,j}$ belong to $g^{-1}B_i$, then all possible splittings produce saddle connections with the same holonomy vectors $v_{i,j}$. Hence all $3^K$ points in the fiber belong to $\widehat{\mathcal H}_{g^{-1}\boldsymbol B}$.

Applying the argument of~\cite[Lemma~7.1]{masur2025lengthssaddleconnectionsrandom}, with the annuli in their computation replaced by the sets $(g^{-1}B_i)$, we obtain
\begin{equation}\label{equ:lift-volume}
    \begin{aligned}
        \widehat\mu_g\left( \widehat{\mathcal{H}}_{g^{-1}\boldsymbol{B}} \right) = 3^K\cdot \mu_{\mathcal{H}'}\left(\mathcal{H}'_{\mathrm{Gen}}\right)\cdot 
        &\left(1 + K\cdot O\left(\frac{1}{g}\right)\right)\cdot\\
        &\prod_{i=1}^k \left(\frac{\mathrm{Leb}(B_i)}{g^2}\right)^{r_i}\cdot |M_{r_1,\ldots,r_k}|.
    \end{aligned}
\end{equation}
Here $\mathcal H'_{\mathrm{Gen}}$ is the generic subset of $\mathcal H'$ defined in~\cite[Definition~6.3]{masur2025lengthssaddleconnectionsrandom}, and the implied constant is independent of $K$ and $g$.

\subsection{Moment computations over the generic subsets}

As in~\cite[Section~8]{masur2025lengthssaddleconnectionsrandom}, we consider the following random variable:
\[
Z_{g,r_1,\ldots,r_k} := \left(g\Lambda_g^+(B_1)\right)_{r_1}\ldots \left(g\Lambda_g^+(B_k)\right)_{r_k}: \mathcal{H}_g\to \mathbb{N}.
\]
Then $Z_{g,r_1,\ldots,r_k}$ counts the number of ordered lists $(V_1,\ldots,V_k)$ where $V_i$ is an ordered $r_i$-tuple of elements in $\Lambda_\omega\cap g^{-1}B_i$. 

\begin{lemma}[{Adaptation of~\cite[Proposition~8.1]{masur2025lengthssaddleconnectionsrandom}}]\label{lem:expectation-on-generic-part}
    For fixed pairwise disjoint bounded subsets $B_1,\ldots,B_k$ in $H$, we have
    \[
    \lim_{g\to\infty}\frac{1}{\mu_g(\mathcal{H}_{\mathrm{Gen}})}\int_{\mathcal{H}_{\mathrm{Gen}}}Z_{g,r_1,\ldots,r_k} = \prod_{i=1}^k \lambda_i^{r_i},\qquad \lambda_i = 16\,\mathrm{Leb}(B_i).
    \]
\end{lemma}

\begin{proof}
    The proof follows that of~\cite[Proposition~8.1]{masur2025lengthssaddleconnectionsrandom}, with the following modifications.

    Set $K:=\sum_{i=1}^k r_i$ and define
    \[
    Z_{g,r_1,\ldots,r_k}^{\mathrm{sc}} := \prod_{i=1}^k\left(N^{\mathrm{sc}}(\omega, g^{-1}B_i)\right)_{r_i}: \mathcal{H}_g\to \mathbb{N}.
    \]
    Thus, $Z_{g,r_1,\ldots,r_k}^{\mathrm{sc}}$ counts ordered lists $\Gamma = (\Gamma_1,\ldots,\Gamma_k)$, where $\Gamma_i$ is an ordered $r_i$-tuple of saddle connections on $\omega$ with holonomy vectors in $g^{-1}B_i$.

    Since $B_i\subset H$ for each $i$, the orientation of each saddle connection in $\Gamma$ is uniquely determined. Hence each list $\Gamma$ counted by $Z_{g,r_1,\ldots,r_k}^{\mathrm{sc}}$ gives rise to a unique
    element of $\widehat{\mathcal{H}}$. In contrast, in~\cite[Proposition~8.1]{masur2025lengthssaddleconnectionsrandom}, the $2^K$ choices of orientations give rise to $2^K$ elements of $\widehat{\mathcal{H}}$.

    The rest of the argument is the same, except that the factor $2^K$ is absent and the areas of the annuli are replaced by $\mathrm{Leb}(B_i)$. Using~\eqref{equ:lift-volume}, we obtain
    \begin{equation}\label{equ:2-K-absent}
    \begin{aligned}
    \int_{\mathcal{H}_{\mathrm{Gen}}}Z^{\mathrm{sc}}_{g,r_1,\ldots,r_k} 
    &= \int_{\widehat{\mathcal{H}}}\boldsymbol{1}\\ 
    &= 3^K
    \cdot 
    \mu_{\mathcal{H}'}\left(\mathcal{H}'_{\mathrm{Gen}}\right)
    \cdot 
    \left(1 + K\cdot O\left(\frac{1}{g}\right)\right)
    \cdot 
    \prod_{i=1}^k \left(\frac{\mathrm{Leb}(B_i)}{g^2}\right)^{r_i}
    \cdot |M_{r_1,\ldots,r_k}|.
    \end{aligned}
    \end{equation}
    According to the computation in~\cite[Proposition~8.1]{masur2025lengthssaddleconnectionsrandom}, 
    \[
    |M_{r_1,\ldots,r_k}| = \frac{(2g-2)!}{(2g-2-2K)!},
    \qquad 
    \lim_{g\to\infty}\frac{\mu_{\mathcal{H}'}\left(\mathcal{H}'_{\mathrm{Gen}}\right)}{\mu_g\left(\mathcal{H}_{\mathrm{Gen}}\right)} 
    = 
    \left(\frac{4}{3}\right)^K.
    \]
    It follows that 
    \[
    \begin{aligned}
        &\lim_{g\to\infty}\frac{1}{\mu_g(\mathcal{H}_{\mathrm{Gen}})}\int_{\mathcal{H}_{\mathrm{Gen}}}Z_{g,r_1,\ldots,r_k}^{\mathrm{sc}}\\ 
        &= \frac{\mu_{\mathcal{H}'}\left(\mathcal{H}'_{\mathrm{Gen}}\right)}{\mu_g\left(\mathcal{H}_{\mathrm{Gen}}\right)}\cdot 3^K\cdot \frac{(2g-2)!}{(2g-2-2K)!}
        \cdot 
        \prod_{i=1}^k \left(\frac{\mathrm{Leb}(B_i)}{g^2}\right)^{r_i}
        \cdot
        \left(1 + K\cdot O\left(\frac{1}{g}\right)\right)\\
        &\longrightarrow 
         \left(\frac{4}{3}\right)^K\cdot 3^K\cdot (2g)^{2K}\cdot \frac{1}{g^{2K}} \cdot 
        \prod_{i=1}^k \mathrm{Leb}(B_i)^{r_i}=  \prod_{i=1}^k \left(16\,\mathrm{Leb}(B_i)\right)^{r_i}.
    \end{aligned}
    \]
    Finally, by Lemma~\ref{lem:two-countings-agree}, $Z_{g,r_1,\ldots,r_k}^{\mathrm{sc}}
    =Z_{g,r_1,\ldots,r_k}$ almost everywhere on $\mathcal{H}_{\mathrm{Gen}}$, which proves the lemma.
\end{proof}

Following~\cite[Section~9]{masur2025lengthssaddleconnectionsrandom}, let
\[
\mu_{\mathrm{Gen}}(A) = \frac{\mu_g(A\cap \mathcal{H}_{\mathrm{Gen}})}{\mu_g(\mathcal{H}_{\mathrm{Gen}})}
\]
for Borel subsets $A\subset \mathcal{H}_g$. By Lemma~\ref{lem:expectation-on-generic-part}
\[
\mathbb{E}_{\mu_{\mathrm{Gen}}}\left[\prod_{i=1}^k\left(g\Lambda_g^+(B_i)\right)_{r_i}\right]\longrightarrow \prod_{i=1}^k \lambda_i^{r_i},\qquad \lambda_i = 16\,\mathrm{Leb}(B_i),
\]
for all $r_1,\ldots,r_k\in\mathbb N$. Here, $\mathbb{E}_{\mu_{\mathrm{Gen}}}$ denotes expectation with respect to $\mu_{\mathrm{Gen}}$. By the method of moments~\cite[Theorem~1.23]{Bollobas2001}, it follows that 
\begin{equation}\label{equ:pp-generic-part}
    \lim_{g\to\infty}\mu_{\mathrm{Gen}}\left(\omega: (g\Lambda^+_g)(B_1)=r_1,\ldots, (g\Lambda^+_g)(B_k)=r_k\right) 
    = 
    \prod_{i=1}^k\frac{\lambda_i^{r_i}e^{-\lambda_i}}{r_i!}.
\end{equation}

It remains to pass from $\mathcal{H}_{\mathrm{Gen}}$ to $\mathcal H_g$. Write
\[
A_g = \left\{\omega\in \mathcal{H}_g: (g\Lambda^+_g)(B_1)=r_1,\ldots, (g\Lambda^+_g)(B_k)=r_k\right\}
\]
Then we have 
\[ 
\left| \mu_g(A_g)-\mu_{\mathrm{Gen}}(A_g) \right| \le \mu_g(\mathcal H_{\mathrm{NG}}). 
\] 
By~\cite[Proposition~6.2]{masur2025lengthssaddleconnectionsrandom}, $\mu_g(\mathcal H_{\mathrm{NG}})\to0$ as $g\to\infty$. Combining this with \eqref{equ:pp-generic-part} proves~\eqref{equ:holonomy_vectors-ppp} and hence Theorem~\ref{thm:holonomy-vectors-ppp}.

\subsection{Extensions to other strata}\label{subsec:other-strata-pp}

In~\cite[Section~10]{masur2025lengthssaddleconnectionsrandom}, Masur--Rafi--Randecker consider two sequences of strata. In both cases, we let $\omega$ be distributed according to the corresponding Masur--Veech measure normalized to have total mass one. For each sequence, we define $\Lambda_g^+$ as above. The preceding generalization applies to both sequences. We record the resulting point process limits.

\begin{theorem}\label{thm:other-strata-ppp}
The following statements hold.

\begin{enumerate}
    \item Let $f,\ell:\mathbb N\to\mathbb R$ satisfy ${f(g)^2\ell(g)}/{g}\longrightarrow 0$.  For any sequence of strata $\mathcal{H}_g^1(m_1,\ldots,m_{\ell(g)},1,\ldots,1)$ satisfying $2\le m_i\leq f(g)$ for each $i$, we have
    \[
        g\Lambda_g^+\xrightarrow{d}\Pi_{16}^H.
    \]

    \item Fix $m\in\mathbb N$, and consider $g\in\mathbb N$ such that $m$ divides $2g-2$. For the sequence of the strata $\mathcal{H}^1_g(m,\ldots,m)$, we have
    \[
        g\Lambda_g^+
        \xrightarrow{d}
        \Pi_{\rho_m}^H,
        \qquad
        \rho_m
        =
        4\left(\frac{m+1}{m}\right)^2.
    \]
\end{enumerate}
\end{theorem}

To obtain Theorem~\ref{thm:other-strata-ppp}, we generalize the two arguments in~\cite[Section~10]{masur2025lengthssaddleconnectionsrandom} in the same way as in the proof of Theorem~\ref{thm:holonomy-vectors-ppp}.

For the first sequence of strata, Masur--Rafi--Randecker show that, under the assumption $\frac{f(g)^2\ell(g)}{g}\to 0$, the contribution of short saddle connections incident to non-simple zeros is negligible, and their argument reduces to the one for the principal stratum. Thus the same generalization as in Theorem~\ref{thm:holonomy-vectors-ppp} gives the Poisson point process of intensity $16$.

For $\mathcal H_g^1(m,\ldots,m)$, we use their corresponding simultaneous collapsing and moment computation. As in the principal case, the full annuli are replaced by arbitrary bounded Borel sets $g^{-1}B_i\subset H$. Since $B_i\subset H$ determines the orientation of each saddle connection, the factor $2^{-K}$ appearing in their moment computation is absent. Keeping all the other factors from their calculation gives $\rho_m=4\left(\frac{m+1}{m}\right)^2$.

\section{Large genus point process and gap asymptotics}

\subsection{Principal strata}

For each $g\ge2$, let $\mathcal H_g:=\mathcal H_g^1(1^{2g-2})$ and write 
\[ 
\xi_g:=\xi_{\mathcal H_g}, \qquad G_{g;\,j_1,\ldots,j_m} := G_{\mathcal H_g;\,j_1,\ldots,j_m} 
\]
for the slope point process and the associated finite-dimensional distribution functions, respectively.

We are now ready to prove the first assertion of Theorem~\ref{thm:large-genus-asymptotics}, which we restate below.

\begin{proposition}\label{prop:large-genus-point-process}
For principal strata $\mathcal H_g$, we have $g^2\xi_{g}  \xrightarrow{d} \Pi_{8}$ as $g\to\infty$.
\end{proposition}

\begin{proof}
    Fix $m\in\mathbb N$, pairwise disjoint intervals $I_i=[a_i,b_i]$, $1\le i\le m$, and $n_1,\ldots,n_m\in\mathbb Z_{\ge0}$. By Lemma~\ref{lem:SV-quadratic-short-holonomy-vectors}, $\Lambda_\omega = \Lambda_\omega^{\mathrm{red}}$ for $\mu_g$-almost every $\omega$. Take $g_{t} = \operatorname{diag}(e^{t/2},e^{-t/2})$ with $t = -2\log g$. Then
    \[
    g_{t}\,T_{g^{-2}I_i} 
    = \left\{
    (x,y)\in\mathbb R^2:
    0< x<\frac{1}{g},\ 
    a_ix\leq y\leq b_ix
    \right\}
    = g^{-1} T_{I_i}.
    \]
    Since $\mu_g$ is invariant under the action of $g_t$, we have
    \[
    \begin{aligned}
    &P_{g^2\xi_g}\left(
    \zeta(I_i)=n_i
    \text{ for }1\leq i\leq m
    \right)\\
    &\qquad=
    \mu_g\left\{\eta\in\mathcal H_g: \#\left(\Lambda_\eta\cap T_{g^{-2}I_i}\right)=n_i
    \text{ for }1\leq i\leq m\right\}\\
    &\qquad=
    \mu_g\left\{
    \eta\in\mathcal H_g:
    \#\left(\Lambda_\eta^+\cap g^{-1}T_{I_i}\right)=n_i
    \text{ for }1\leq i\leq m
    \right\}\\
    &\qquad=
    \mu_g\left\{
    \eta\in\mathcal H_g:
    \#\left(g\Lambda_\eta^+\cap T_{I_i}\right)=n_i
    \text{ for }1\leq i\leq m
    \right\}.
    \end{aligned}
    \]
    Note that by $T_{I_i}\subset H$ , the restriction from $\Lambda_\eta$ to $\Lambda_\eta^+$ introduces no change in the corresponding counts.
    
    By Theorem~\ref{thm:holonomy-vectors-ppp} and by taking $B_i = T_{I_i}$, $1\le i\le m$, in~\eqref{equ:holonomy_vectors-ppp}, we obtain
    \[
        \lim_{g\to\infty}
        P_{g^2\xi_g}\left(
        \zeta(I_i) = n_i
        \text{ for }1\leq i\leq m
        \right) = \prod_{i=1}^m\frac{\lambda_i^{n_i}e^{-\lambda_i}}{n_i!},
    \]
    where $\lambda_i = 16\,\mathrm{Leb}(T_{I_i}) = 8\,|I_i|$. It follows from Theorem~\ref{thm:characterization-point-process-convergence} that $g^2\,\xi_g\xrightarrow{d}\Pi_8$ as $g\to\infty$.
\end{proof}

To pass from the above point process convergence to the convergence of the corresponding Palm distributions, we need convergence of the intensities.

\begin{lemma}\label{lem:large-genus-convergence-of-intensities}
    For the principal strata $\mathcal{H}_g$, the intensity of $g^2\xi_g$ converges to $8$, the intensity of the Poisson point process $\Pi_8$, as $g\to\infty$.
\end{lemma}

\begin{proof}
    By Lemma~\ref{lem:slope-point-process}, the intensity of $g^2\xi_g$ is $c^{\mathrm{red}}(\mathcal H_g)/(2g^2)$. By Lemma~\ref{lem:SV-quadratic-short-holonomy-vectors}, $c^{\mathrm{red}}(\mathcal H_g) = c(\mathcal H_g)$ for Masur--Veech measures. Since~\eqref{equ:Siegel-Veech-asymptotic} gives $c(\mathcal H_g)/2\sim 8g^2$, the result follows.
\end{proof}

We next prove the second assertion of Theorem~\ref{thm:large-genus-asymptotics}.

\begin{proposition}\label{prop:large-genus-approximation}
    For every $m\in\mathbb N$ and every collection of distinct
    $j_1,\ldots,j_m\in\mathbb Z$,
    \[
    \lim_{g\to\infty}
    \sup_{t_1,\ldots,t_m\ge 0}
    \left|
    G_{g;\,j_1,\ldots,j_m}
    \left(
    t_1,\ldots,t_m
    \right)
    -
    \prod_{\ell=1}^m\left(1-e^{-8g^2t_\ell}\right)
    \right|
    =0.
    \]
\end{proposition}

\begin{proof}
    Proposition~\ref{prop:large-genus-point-process} implies that $g^2\xi_g \xrightarrow{d} \Pi_8$ as $g\to\infty$. By Lemma~\ref{lem:large-genus-convergence-of-intensities}, their intensities also converge. Hence, the assumptions of Lemma~\ref{lem:continuous-mapping-theorem} are satisfied, and we have the weak convergence:
    \[
    \Delta_\ast P_{g^2\xi_{g}}^0
    \longrightarrow
    \Delta_\ast \Pi_8
    =
    \bigotimes_{ j\in\mathbb Z}
    \mathrm{Exp}(8),
    \quad\text{as } g\to\infty,
    \]
    where the last equality uses Theorem~\ref{thm:gap-distributions-ppp}. Since the joint distribution function of $\bigotimes_{j\in\mathbb Z}\mathrm{Exp}(8)$ is continuous, the Portmanteau theorem gives, for any $t_1,\ldots,t_m\ge0$,
    \[
    G_{g;\, j_1,\ldots,j_m}\left(g^{-2}t_1,\ldots,g^{-2}t_m\right) 
    \longrightarrow \prod_{\ell=1}^m\left(1-e^{-8t_\ell}\right)
    \]
    as $g\to\infty$. Moreover, by the multivariate version of P\'olya's theorem~\cite[Example 6 in Section 9]{BillingsleyTopsoe1967Uniformity}, this convergence is uniform on $[0,\infty)^m$. Replacing each $t_\ell$ by $g^2t_\ell$ proves the desired convergence.
\end{proof}

Proposition~\ref{prop:large-genus-point-process} and Proposition~\ref{prop:large-genus-approximation} together imply Theorem~\ref{thm:large-genus-asymptotics}.

\subsection{Other sequences of strata}\label{subsec:other-strata-gap}

Finally, we extend the two large genus results above to the two sequences of strata considered in Theorem~\ref{thm:other-strata-ppp}. As before, we let $\omega$ be distributed according to the corresponding Masur--Veech measure normalized to have total mass one.

For the sequence of strata $\mathcal H_g^1(m_1,\ldots,m_{\ell(g)},1,\ldots,1)$ in Theorem~\ref{thm:other-strata-ppp}(1), we write $\xi_g^{(f,\ell)}$ for the corresponding slope point process and $G_{g;\,j_1,\ldots,j_r}^{(f,\ell)}$ for its finite-dimensional gap distribution functions. For the sequence $\mathcal H_g^1(m,\ldots,m)$ in Theorem~\ref{thm:other-strata-ppp}(2), we similarly write $\xi_g^{(m)}$ and $G_{g;\,j_1,\ldots,j_r}^{(m)}$.

\begin{theorem}\label{thm:other-strata-gap-asymptotics}
The following statements hold.

\begin{enumerate}
    \item Let $f:\mathbb N\to\mathbb R_{>0}$ and $\ell:\mathbb N\to\mathbb N$ satisfy $f(g)^2\ell(g)/g\to0$. For any sequence of strata $\mathcal H_g^1(m_1,\ldots,m_{\ell(g)},1,\ldots,1)$ satisfying $2\le m_i\leq f(g)$ for each $i$, we have
    \[
        g^2\xi_g^{(f,\ell)}\xrightarrow{d}\Pi_8.
    \]
    Moreover, for every $r\in\mathbb N$ and any distinct $j_1,\ldots,j_r\in\mathbb Z$,
    \[
        \lim_{g\to\infty}
        \sup_{t_1,\ldots,t_r\geq0}
        \left|
        G_{g;\,j_1,\ldots,j_r}^{(f,\ell)}(t_1,\ldots,t_r)
        -
        \prod_{q=1}^m\left(1-e^{-8g^2t_q}\right)
        \right|
        =0.
    \]

    \item Fix $m\in\mathbb N$, and consider $g\in\mathbb N$ such that $m$ divides $2g-2$. For the sequence of strata $\mathcal H_g^1(m,\ldots,m)$, we have
    \[
        g^2\xi_g^{(m)}
        \xrightarrow{d}
        \Pi_{\rho_m},
        \qquad
        \rho_m=4\left(\frac{m+1}{m}\right)^2.
    \]
    Moreover, for every $r\in\mathbb N$ and any distinct $j_1,\ldots,j_r\in\mathbb Z$,
    \[
        \lim_{g\to\infty}
        \sup_{t_1,\ldots,t_r\geq0}
        \left|
        G_{g;\,j_1,\ldots,j_r}^{(m)}(t_1,\ldots,t_r)
        -
        \prod_{q=1}^r\left(1-e^{-\frac{\rho_m}{2} g^2t_q}\right)
        \right|
        =0.
    \]
\end{enumerate}
\end{theorem}

As in the principal stratum case, to pass from the point process limits to the gap asymptotics, we also need the corresponding convergence of intensities.

\begin{lemma}\label{lem:other-strata-convergence-of-intensities}
For the two sequences of strata above, the intensities of $g^2\xi_g^{(f,\ell)}$ and $g^2\xi_g^{(m)}$ tend to $8$ and $\rho_m/2$, respectively.
\end{lemma}

\begin{proof} 
By Lemma~\ref{lem:slope-point-process} and Lemma~\ref{lem:SV-quadratic-short-holonomy-vectors}, the intensities of $g^2\xi_g^{(f,\ell)}$ and $g^2\xi_g^{(m)}$ are the corresponding Siegel--Veech constants divided by $2g^2$. 

For the first sequence, the number of zeros is $\ell(g) + 2g-2-\sum_{i=1}^{\ell(g)}m_i$. Then~\eqref{equ:Siegel-Veech-asymptotic} gives
\[
\frac{c(\mathcal H_g^1(m_1,\ldots,m_{\ell(g)},1,\ldots,1))}{2}\sim \frac{1}{2}\left(4g-4-\sum_{i=1}^{\ell(g)}(m_i-1)\right)^2.
\]
Since $2\le m_i\le f(g)$ and $f(g)^2\ell(g)/g\to0$, $\sum_{i=1}^{\ell(g)}(m_i-1)\leq f(g)^2\ell(g) = o(g)$. Thus, the right-hand side is asymptotic to $8g^2$.

For $\mathcal H_g^1(m,\ldots,m)$, the number of zeros is $(2g-2)/m$, and hence~\eqref{equ:Siegel-Veech-asymptotic} gives 
\[ 
\frac{c(\mathcal H_g^1(m,\ldots,m))}{2} \sim 2\left(\frac{m+1}{m}\right)^2g^2 = \frac{\rho_m}{2}g^2.\]
The result follows.
\end{proof}

\begin{proof}[Proof of Theorem~\ref{thm:other-strata-gap-asymptotics}]
The point process convergences follow from Theorem~\ref{thm:other-strata-ppp} by the same $g_t$-rescaling argument as in the proof of Proposition~\ref{prop:large-genus-point-process}. Together with Lemma~\ref{lem:other-strata-convergence-of-intensities}, Lemma~\ref{lem:continuous-mapping-theorem} then gives the corresponding convergence of the gap sequences to
\[
    \bigotimes_{j\in\mathbb Z}\mathrm{Exp}(8)
    \qquad\text{and}\qquad
    \bigotimes_{j\in\mathbb Z}\mathrm{Exp}\left(\frac{\rho_m}{2}\right),
\]
respectively. The asserted uniform convergence of the finite-dimensional gap distribution functions follows from the multivariate P\'olya theorem exactly as in the proof of Proposition~\ref{prop:large-genus-approximation}.
\end{proof}

\newcommand{\etalchar}[1]{$^{#1}$}


\end{document}